\documentclass[10pt,reqno]{amsart}
     \makeatletter
     \def\section{\@startsection{section}{1}%
     \z@{.7\linespacing\@plus\linespacing}{.5\linespacing}%
     {\bfseries%\normalfont\scshape
     \centering
     }}
     \def\@secnumfont{\bfseries}
     \makeatother
\newtheorem{theorem}{Theorem}[section]

\theoremstyle{definition}
\newtheorem{definition}[theorem]{Definition}

\theoremstyle{remark}

\numberwithin{equation}{section}
\usepackage{graphicx}      % Allows inclusion of figures
\usepackage{float}            % Allows figures to not float to next page
\usepackage{amsmath}  	% \text{} etc
\usepackage{latexsym}	% LaTeX fonts
\usepackage{amssymb}  	% for \feedback symbol; superset of amsfonts
\usepackage{amsaddr}      % used to get proper journal addresses
\usepackage{amsthm}   	% AMS style theorem and proof environments

\usepackage[T1,hyphens]{url}

\usepackage[x11names]{xcolor}
\usepackage[colorlinks=true,linkcolor=blue,citecolor=SpringGreen4,urlcolor=blue]{hyperref} % Allows hyperlinks

\DeclareMathOperator{\sgn}{sgn}
\makeatletter
\renewcommand{\tocsection}[3]{%
  \indentlabel{\@ifnotempty{#2}{\bfseries\ignorespaces#1 #2\quad}}\bfseries#3}
\renewcommand{\tocsubsection}[3]{%
  \indentlabel{\@ifnotempty{#2}{\ignorespaces#1 #2\quad}}#3}
\newcommand\@dotsep{4.5}
\def\@tocline#1#2#3#4#5#6#7{\relax
  \ifnum #1>\c@tocdepth % then omit
  \else
    \par \addpenalty\@secpenalty\addvspace{#2}%
    \begingroup \hyphenpenalty\@M
    \@ifempty{#4}{%
      \@tempdima\csname r@tocindent\number#1\endcsname\relax
    }{%
      \@tempdima#4\relax
    }%
    \parindent\z@ \leftskip#3\relax \advance\leftskip\@tempdima\relax
    \rightskip\@pnumwidth plus1em \parfillskip-\@pnumwidth
    #5\leavevmode\hskip-\@tempdima{#6}\nobreak
    \leaders\hbox{$\m@th\mkern \@dotsep mu\hbox{.}\mkern \@dotsep mu$}\hfill
    \nobreak
    \hbox to\@pnumwidth{\@tocpagenum{\ifnum#1=1\bfseries\fi#7}}\par% <-- \bfseries for \section page
    \nobreak
    \endgroup
  \fi}
\AtBeginDocument{%
\expandafter\renewcommand\csname r@tocindent0\endcsname{0pt}
}
\def\l@subsection{\@tocline{2}{0pt}{2.5pc}{5pc}{}}
\makeatother
\begin{document}

\title[Iterated Logarithm Bounds of BGC Stochastic Processes]{Iterated Logarithm Bounds of Bi-Directional Grid Constrained Stochastic Processes}

\author[Aldo Taranto, Shahjahan Khan, Ron Addie]{Aldo Taranto, Shahjahan Khan, Ron Addie}
\email{Aldo.Taranto@, Shahjahan.Khan@, Ron.Addie@, @usq.edu.au}



%\maketitle
\maketitle %This makes the title page

%\author{Aldo Taranto $\quad$ , $\quad$ Ron Addie $\quad$ , $\quad$ Shahjahan Khan
%\author{Aldo Taranto 
%$\quad$ , $\quad$ Laura L. Sacerdote
%}

%\email{Aldo.Taranto@}
%\address{Mathematics \& Statistics, University of Southern Queensland, Toowoomba, QLD 4350, Australia}

%\vspace*{1cm}

%\curraddr{Credience Corporation, www.credience.com}
%\curremail{Aldo.Taranto@credience.com}

%\vspace*{1cm}

%\author{Ron Addie}
%\email{Ron.Addie@} %<----------------ADD BACK IN

%\author{Shahjahan Khan}
%\email{Aldo.Taranto@ \quad \quad @usq.edu.au}
%\email{Aldo.Taranto@ $\quad$ , $\quad$ Ron.Addie@ $\quad$ , $\quad$ Shahjahan.Khan@ $\quad$ , $\quad$ @usq.edu.au
% , laura.sacerdote@unito.it
%\urladdrname{www.usq.edu.au}

%    Information for third author
%\author{Ravinesh C. Deo}
%\email{Ravinesh.Deo@usq.edu.au}
%\vspace*{0.5cm}

%\thanks{We would like to thank Prof. Martin Schweizer of ETH Z\"{u}rich and A/Prof. Ron Addie of University of Southern Queensland for their invaluable advice on refining this paper.}

%\thanks{Support information for the third author.}
%\address{\text{ }\\
%Dipartimento di Matematica\\
%Universit\'{a} di Torino\\
%10123 Torino, Italia}

%    General info
\date{\today}


\begin{abstract}
We derive a novel framework called Bi-Directional Grid Constrained (BGC) stochastic processes in which the further an It\^{o} diffusion drifts away from the origin, then the further it will be constrained.
By making suitable modifications to the Law of Iterated Logarithm (LIL), we derive a novel theorem about the upper and lower bounds for BGC processes and their hidden barrier.
To visualize the theorem, we run many simulations of the It\^{o} diffusions for a representative expression for $\lambda (X, t)$, both before and after BGC and uncover some interesting results with applications into finance and many other areas.
\end{abstract}

%%%%%%%%%%%%%%%%%%%%%%%%%%%%%%%%%%%%%%%%%%%%%%%%%%%%%%%%%%%%%%%
%\newpage
%\tableofcontents %<------------------------------------------------------EDITOR ASKED TO REMOVE THIS

%\vfill
%\vspace*{0.5cm}

%\bigskip
%%%%%%%%%%%%%%%%%%%%%%%%%%%%%%%%%%%%%%%%%%%%%%%%%%%%%%%%%%%%%%%

%\newpage

%\section*{Terminology}

%\begin{table*}[h]
 %  \begin{center}
%\begin{tabular}{  l  l  }
%\hline
% \textbf{Term} & \textbf{Description}   \\ \hline
%BGC                & Bi-Directional Grid Constrained\\
%$X_t$              & Stochastic process over time $t$\\
%$B_t$              & Brownian motion over time $t$ \\ 
%$W_t$              & Wiener process over time $t$ \\ 
%$T$                 & Time\\
%$\mu(X, t)$             & Drift coefficient of $X$ over $t$\\
%$\sigma(X, t)$         & Diffusion coefficient of $X$ over $t$\\
%$|| x ||$          & Norm of $x$ in Hilbert space\\
%$| x |$            & Absolute value of $x$\\
%$\blacksquare$ & End of theorem or corollary (not proved), Definition\\
%$\square$ & End of proof\\
%$X \perp Y$     & $X$ and $Y$ are perpendicular\\
%$\langle \bigcdot | \bigcdot \rangle$ & Hermitian inner product operator\\
%$\mathfrak{Re}(\bigcdot)$ & Real component of a complex number\\
%$\mathfrak{Im}(\bigcdot)$ & Imaginary component of a complex number\\
%\hline
%\end{tabular}
   %\label{Tab:Abbreviations}
%   \end{center}
%\end{table*}

\bigskip
%%%%%%%%%%%%%%%%%%%%%%%%%%%%%%%%%%%%%%%%%%%%%%%%%%%%%%%%%%%%%%%
\section{Introduction}

\noindent
The problem that this paper solves is the identification of a novel class of It\^{o} diffusions and their corresponding formulation, in which the further they drift away from the origin, then the more resistence and hence constraining they undergo. 
We examine the constraining of stochastic processes by subtle perturbations rather than by the usual direct perturbations, such as through the use of hard barriers.
One of the earliest forms of subtle constraining is the Langevin equation of Physics (Langevin, 1908) \cite{Langevin1908}.
The Langevin equation is a stochastic differential equation (SDE) that describes a particular form of Brownian motion, the apparently random movement of a particle in a fluid due to its collisions with other particles in the fluid and is expressed as,

\[
\displaystyle m \, {\frac {d \mathbf{v} }{dt}} \quad = \underbrace{\strut  \boldsymbol{\eta}\left(t\right)}_{\text{stochastic term}}
  -  \underbrace{\strut \lambda \mathbf{v}}_{\text{constraining term}},
\]

\bigskip \noindent
where $m$ is the mass of the particle, $\mathbf{v}$ is its velocity and $t$ is time.
The force acting on the particle is written as a sum of a viscous force proportional to the particle's velocity (Stokes' law), and a noise term $\displaystyle{\boldsymbol{\eta }}\left(t\right)$, noting that $\mathbf{v}$ and $\boldsymbol{\eta}$ are vectors.
If the random $\displaystyle{\boldsymbol{\eta }}\left(t\right)$ term was not present, then the above equation would simply be a partial differential equation (PDE). 
The dampening term $-\lambda \mathbf{v}$ constrains or limits the movement of the particle, as shown in Figure \ref{Fig:TwoTypesofIncrementalGradientsforItoDiffusions}(b).

\bigskip \noindent
There are some unwanted complexities for our research in dealing with multiple particles, namely the interaction of every particle with its nearest neighbouring particles.
Hence, we focus on the simpler case of a 1-Dimensional It\^{o} diffusion.

\bigskip \noindent
Turning to Chemistry and Biochemistry, further more relevant examples of subtle constraining of a stochastic process include concentration gradients, in which a molecule diffuses within a medium which becomes increasingly more concentrated.
In Biology, the diffusion of particles such as nutrients and minerals, through a porous membrane can become increasingly constrained the further they pass through the membrane.
We note though that these various fields require numerous variables such as temprature, and so are too limitted for our mathematical purposes. 

\bigskip
\begin{definition}
\textbf{Bi-Directional Grid Constrained (BGC) Stochastic Process}.
A BGC stochastic process for a random variable $X$ over time $t$ is one in which the further it departs from the origin, then the further it will be constrained from above and below (bi-directionally) along that $X$ dimension.
\end{definition}

\bigskip \noindent
This BGC defnition is expressed more precisely as an SDE in (\ref{Eq:BGC}) and has been illustrated in Figure \ref{Fig:TwoTypesofIncrementalGradientsforItoDiffusions}(a).

\begin{figure}
   \centering
 \includegraphics[scale=0.7]{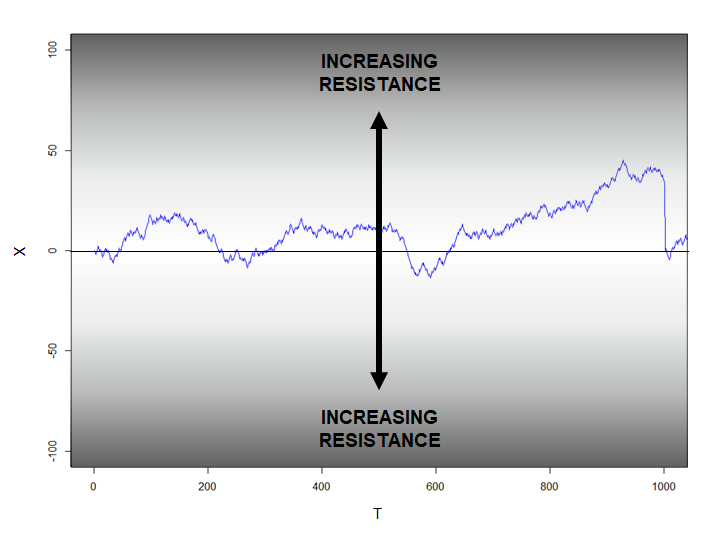}\\
\textbf{\footnotesize \noindent
(a). Vertical Gradient(s) - (Bi-Directional)}\\
   \includegraphics[scale=0.7]{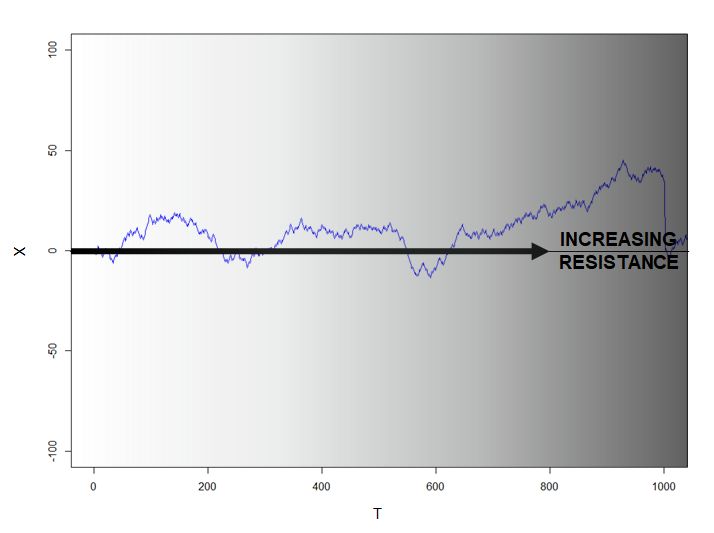}\\
\textbf{\footnotesize \noindent
(b). Horizontal Gradient - (Uni-Directional)}\\
   \caption{Two Types of Incremental Gradients for It\^{o} Diffusions}
   \label{Fig:TwoTypesofIncrementalGradientsforItoDiffusions}
\flushleft
\textbf{\footnotesize \noindent
(a). It\^{o} diffusion is constrained from above and below (bi-directionally), the more it vertically drifts away from the origin over time.\\ 
(b). It\^{o} diffusion is constrained to the right (uni-directionally), the more it horizontally drifts away from the origin over time.
}
\end{figure}
%\FloatBarrier

\bigskip \noindent
Before progressing further, it is important to define the use of two barriers per dimension of the It\^{o} diffusion, as shown in Figure \ref{Fig:BarrierOrientationandTypeofStochasticProcess}.

\begin{figure}
   \centering
 \includegraphics[width=\linewidth]{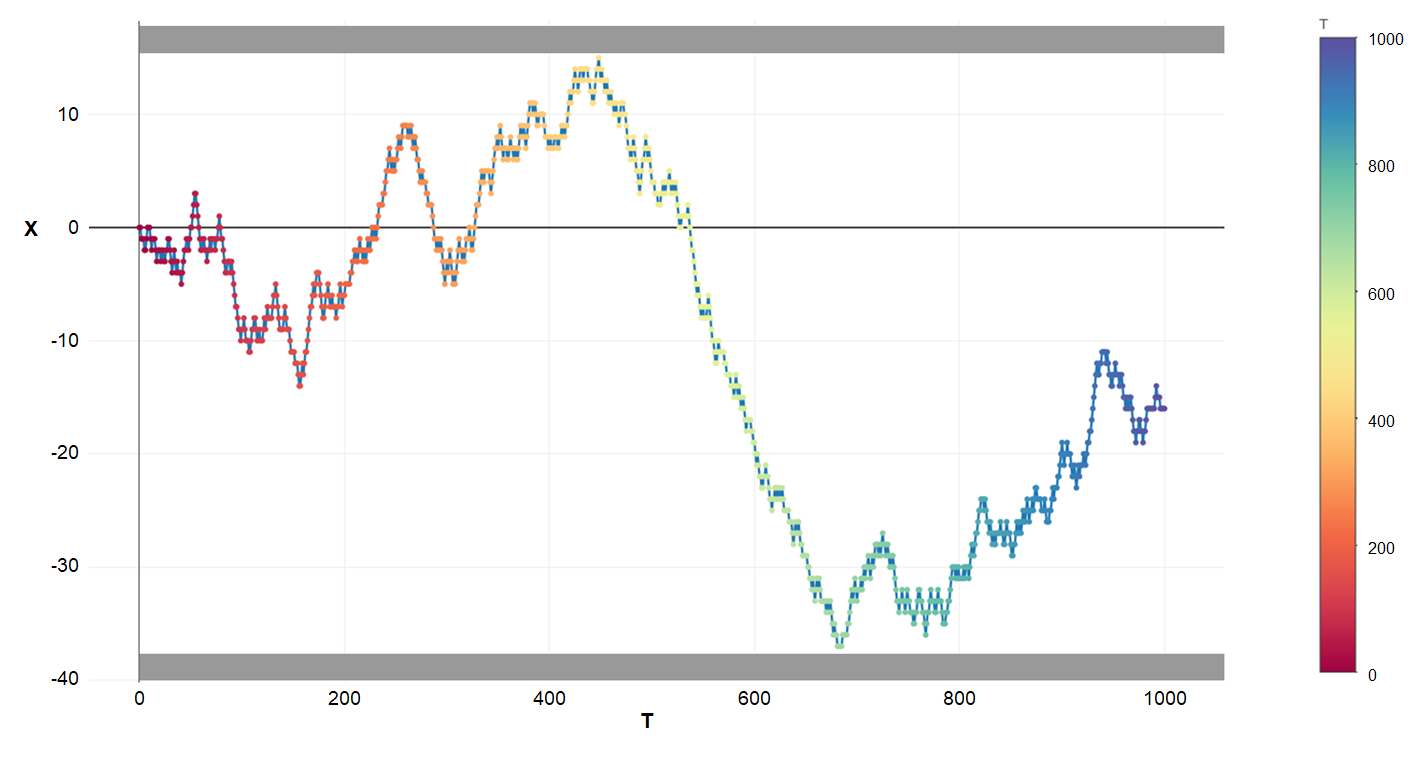}\\
\textbf{\footnotesize \noindent
(a). Horizontal Reflective Barriers}\\
$\quad$ \\
   \includegraphics[width=\linewidth]{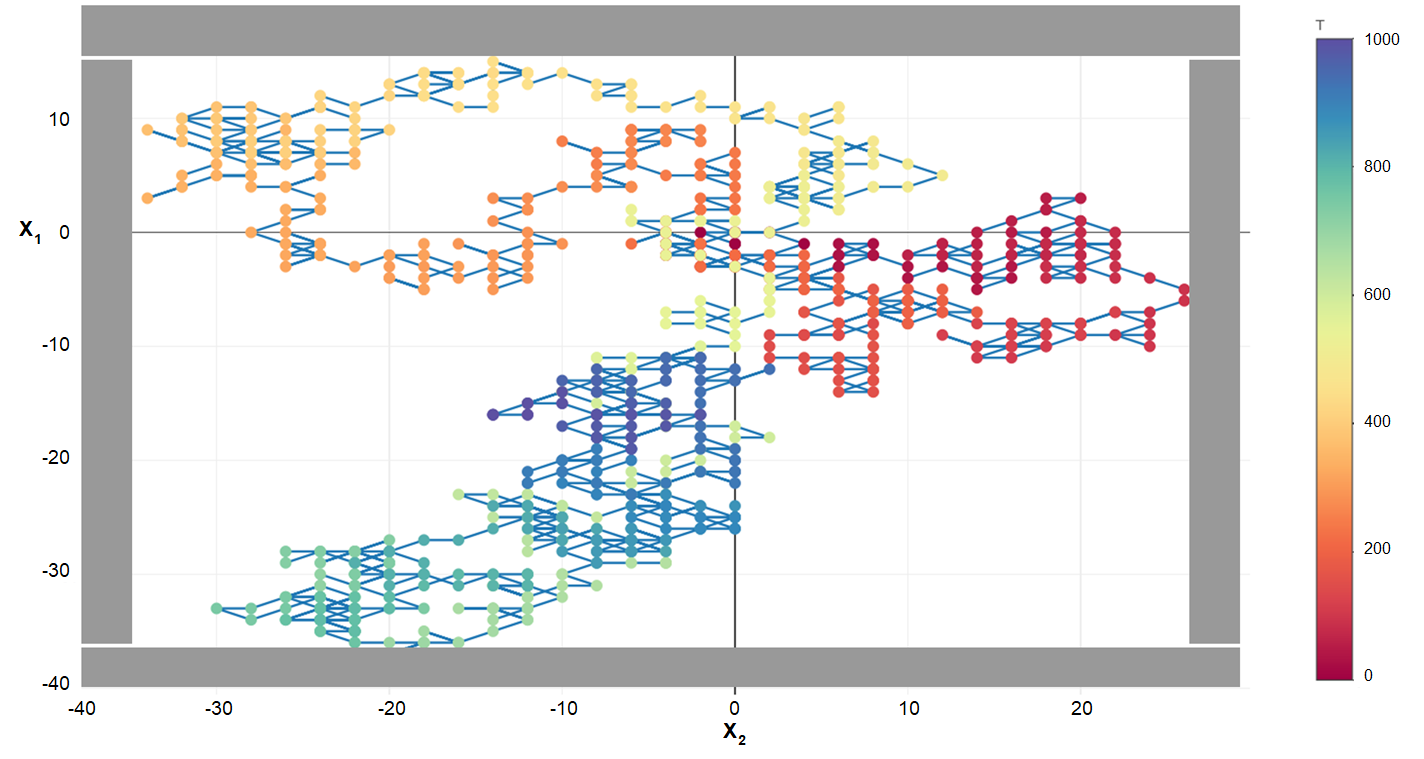}\\
\textbf{\footnotesize \noindent
(b). Horizontal \& Vertical Reflective Barriers}\\
   \caption{Barrier Orientation and Dimension of It\^{o} Diffusions}
   \label{Fig:BarrierOrientationandTypeofStochasticProcess}
\flushleft
\textbf{\footnotesize \noindent 
As the It\^{o} diffusions evolve over 1,000 time steps, we colorize each step according to its index in relation to the hot to cold colour scale.\\
(a). Colorization is not required (but still applied) as the time axis is explicit and so the evolution of the 1-Dimensional It\^{o} diffusion is clear.\\
(b). Colorization helps visualise the evolution of the 2-Dimensional It\^{o} diffusion over time, noting that we are now looking ``top down'' along the implicit time axis.}
\end{figure}
%\FloatBarrier

\bigskip \noindent
Figure \ref{Fig:BarrierOrientationandTypeofStochasticProcess} shows and as will become more apparent in the Methodology section that BGC is an `infinitely scalable process' (i.e. can be generalized to higher dimensions) via the use of $n$-Dimensional It\^{o} processes, but it will suffice for the purposes of this paper to revert to simply 1-Dimensional It\^{o} processes.
Whilst BGC involves more subtle barriers than these hard barriers, Figure \ref{Fig:BarrierOrientationandTypeofStochasticProcess} illustrates that two barriers are required per dimension for Bi-Directional Grid Constraining to occur, as every dimension of possible movement requires to be constrained above and below the origin.
It is also worthwhile defining the types of barriers in Figure \ref{Fig:TheThreeMainTypesofBarriers}.

\begin{figure}
   \centering
   \includegraphics[scale=0.37]{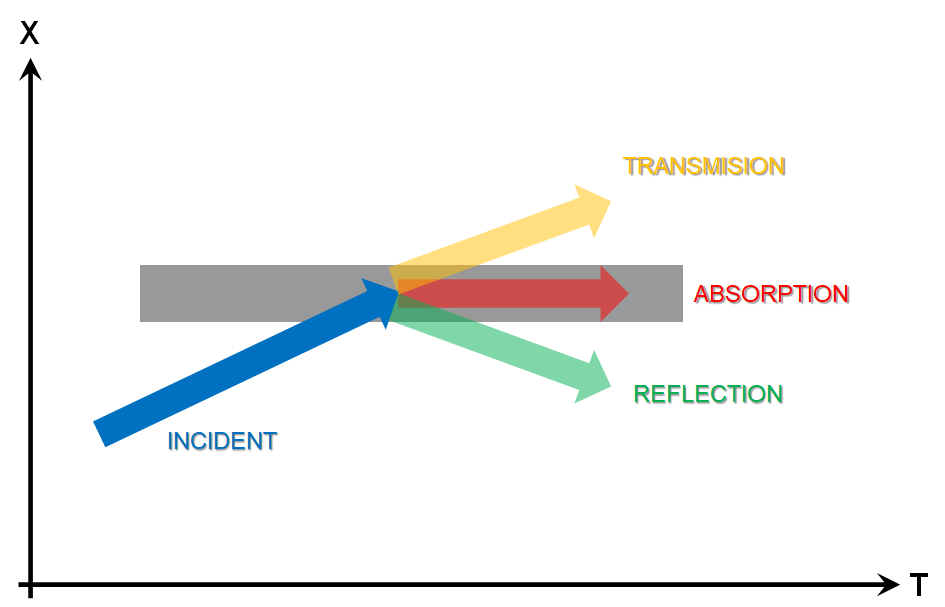}
   \includegraphics[scale=0.37]{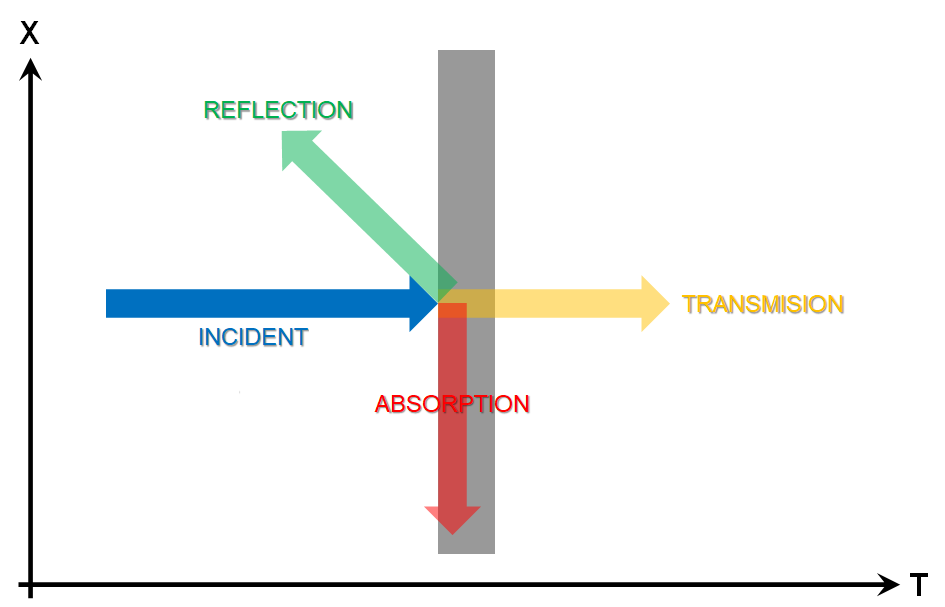}\\
\textbf{\footnotesize \noindent
(a). Horizontal Barrier \quad \quad \quad \quad \quad \quad \quad \quad  (b). Vertical Barrier
}\\
   \caption{Three Main Types of Hard Barriers and Their Two Main Orientations}
   \label{Fig:TheThreeMainTypesofBarriers}
\flushleft
\textbf{\footnotesize \noindent
(a). The 1-Dimensional It\^{o} diffusion can be used since the barrier is horizontal.\\
(b). The only way a reflection from a vertical barrier can occur is if an $\boldsymbol{n}$-Dimensional It\^{o} diffusion is used, where $\boldsymbol{n \geq 2}$ (otherwise it would be travelling backwards in time).
}
\end{figure}
%\FloatBarrier

\bigskip \noindent
We now begin formalizing a `first principles' approach to finding a mathematical expression for BGC, and one such approach could be to consider each equally spaced vertical level or graduation mark of Figure \ref{Fig:ARandomWalkona1DimensionalLatticeChain} to be rotated by $90^{\circ}$, to behave as a horizontal barrier in which the first barriers closest to the origin are fully transmissive barriers.
As the It\^{o} diffusion reaches the next horizontal barrier at some equally spaced interval $\Delta x$ (hence the usage of `grid' for $\mathbb{R}^{n \geq 2}$), then it becomes less transmissive and more reflective.
Ultimately, there will be one uppermost and one lowermost barrier that will be fully reflective.

\bigskip \noindent
Let $\mathbb{P}(x,t)$ be the probability that the random variable $X$ is at position $x$ at time $t$.
Then according to the rule of movement of the variable, there are only two possibilities that it reaches position $x$ at time $t +\Delta t$.
The variable was either at $x-\Delta x$ at time $\Delta t$ and jumped to the right; or the variable was at $x+\Delta x$ at time $\Delta t$ and jumped to the left as shown in Figure \ref{Fig:ARandomWalkona1DimensionalLatticeChain}.

\begin{figure}
   \centering
   \includegraphics[scale=0.58]{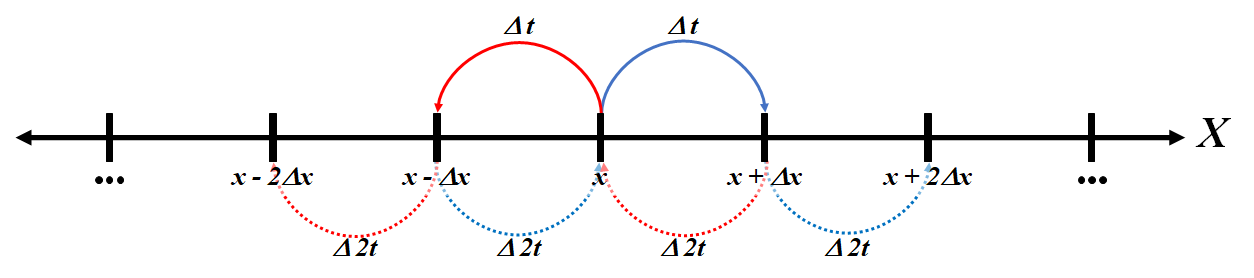}\\
   \caption{A Random Walk on a 1-Dimensional Lattice (Chain)}
   \label{Fig:ARandomWalkona1DimensionalLatticeChain}
\textbf{\footnotesize \noindent
$\boldsymbol{X}$ = Distance $\quad$ , $\quad$ $\boldsymbol{t \in [0,T]}$ = Time\\
Solid lines = $\boldsymbol{t + \Delta t}$ $\quad$ , $\quad$ Dotted lines = $\boldsymbol{t + 2 \Delta t}$}
\end{figure}

\bigskip \noindent
Since the next movement of the variable is independent of its present location, the probability that the variable is at position $x$ at time $t + \Delta t$ given that it was at
position $x - \Delta x$ at time $t$ is $\frac{1}{2} \mathbb{P}(x -\Delta x,t)$, while the probability that the variable is at position $x$ at time $t +\Delta t$ given that it was at $x+\Delta x$ at time $t$ is $\frac{1}{2}\mathbb{P}(x+\Delta x,t)$.
Thus, the probability that the variable is at a position at a point in time is given as follows for the first two cases.

%\bigskip
\newpage
\noindent
At $t + \Delta t$:

\bigskip \noindent
\begin{equation}
 \mathbb{P}(x, t + \Delta t) = \frac{1}{2}  \mathbb{P}(x - \Delta x, t) + \frac{1}{2}  \mathbb{P}(x + \Delta x, t). \nonumber
\end{equation}

\bigskip \noindent
At $t + 2 \Delta t$:

\bigskip \noindent
\begin{eqnarray*}
 \mathbb{P}(x, t + 2 \Delta t)
  & = & \frac{1}{2}  \mathbb{P}(x - 2\Delta x, t) + \frac{1}{2}  \mathbb{P}(x + 2\Delta x, t) \\
  &     &+ \frac{1}{2}  \mathbb{P}(x - \Delta x, t) + \frac{1}{2}  \mathbb{P}(x + \Delta x, t).
\end{eqnarray*}

\bigskip \noindent
By proceeding in this manner, one can amass a series of difference equations, each for different time periods and then combine them, some having terms that cancel out, to derive the general equation.
Khantha \& Balakrishnan (1983) \cite{KhanthaBalakrishnan1983} and the references therein extend such an approach to have biassed random variables where the probability of moving to the left is not the same as the probability of moving to the right.
These authors also introduce various barriers in the framework of lattices.
However, this approach assumes that the direct constraining is due to either 100\% absorptive barriers, 100\% reflective barriers or 100\% transmissive barriers.
It would be much more difficult to formulate the subtle gradual constraining required for of BGC stochastic processes in the same lattice framework.
We thus turn to It\^{o} calculus for a more mathematically elegant approach to defining and assessing the true impact of BGC on stochastic processes.

\bigskip \noindent
When we examine numerous It\^{o} diffusions, we know that a crude approximation of the paths' bounds is $\pm \sqrt{t}$.
A far more accurate estimate for the upper and lower bounds is the Law of the Iterated Logarithm (LIL) in the formulation proved by Kolmogorov  (1929) \cite{KolmogorovLIL}, which is referenced in the Literature Review and stated in the Methodology section.
This is shown in Figure \ref{Fig:LawofIteratedLogarithmSimulated}.

\begin{figure}
   \centering
   \includegraphics[width=\linewidth]{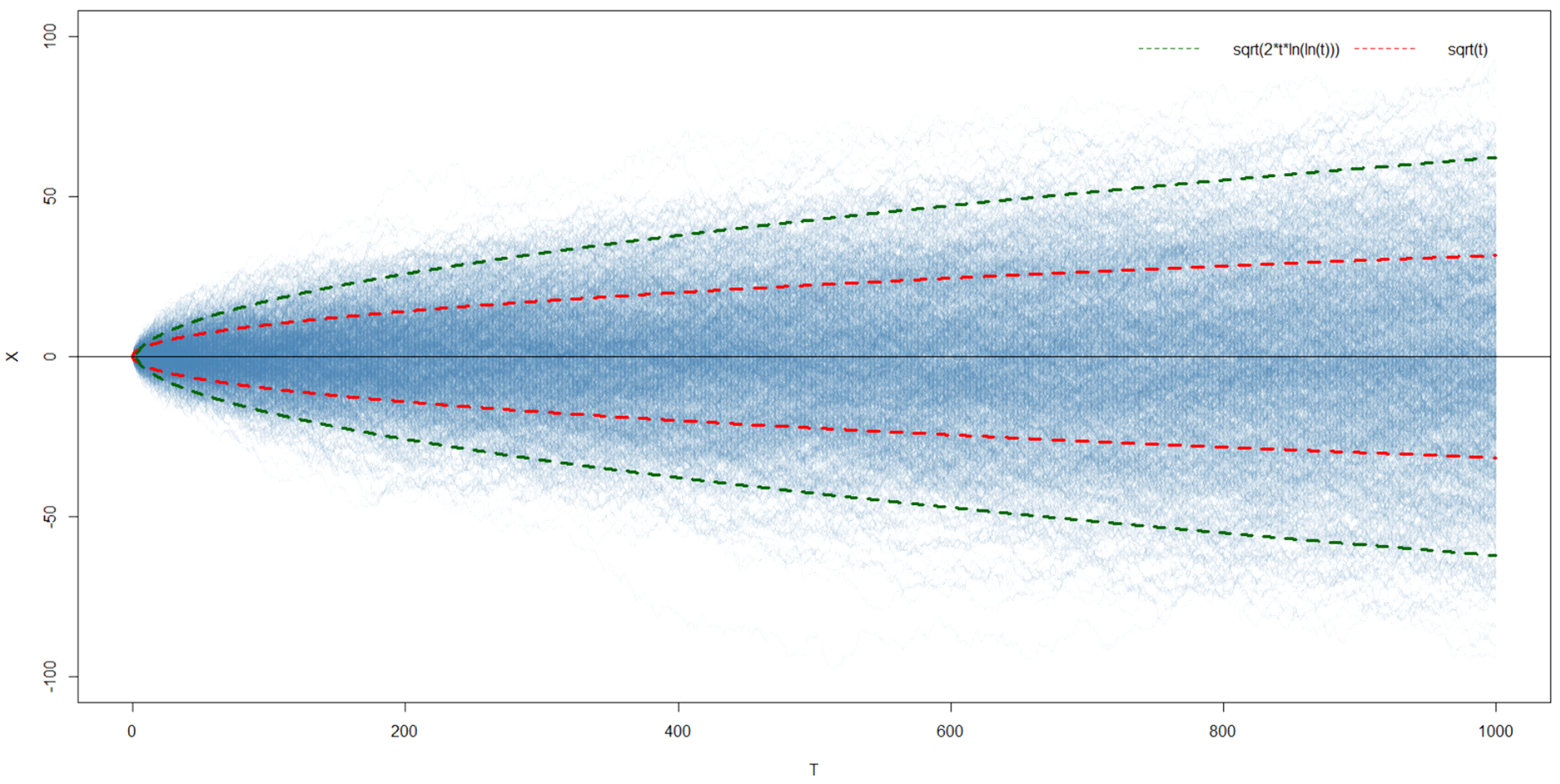}
   \caption{Law of Iterated Logarithm Simulated}
   \label{Fig:LawofIteratedLogarithmSimulated}
\textbf{\footnotesize \noindent
Green envelope = $\boldsymbol{\sqrt{2t \ln(\ln(t))}}$ $\quad$ , $\quad$ Red envelope = $\boldsymbol{\sqrt{t}}$.
}
\end{figure}
%\FloatBarrier

\bigskip \noindent
It is worthwhile noting that when we zoom into Figure \ref{Fig:LawofIteratedLogarithmSimulated}, the LIL is undefined in $[0,e)$ and also that the LIL soon overtakes $\sqrt{t}$, as shown in Figure \ref{Fig:LawofIteratedLogarithmSimulatedZoomed}.

\begin{figure}
   \centering
   \includegraphics[width=\linewidth]{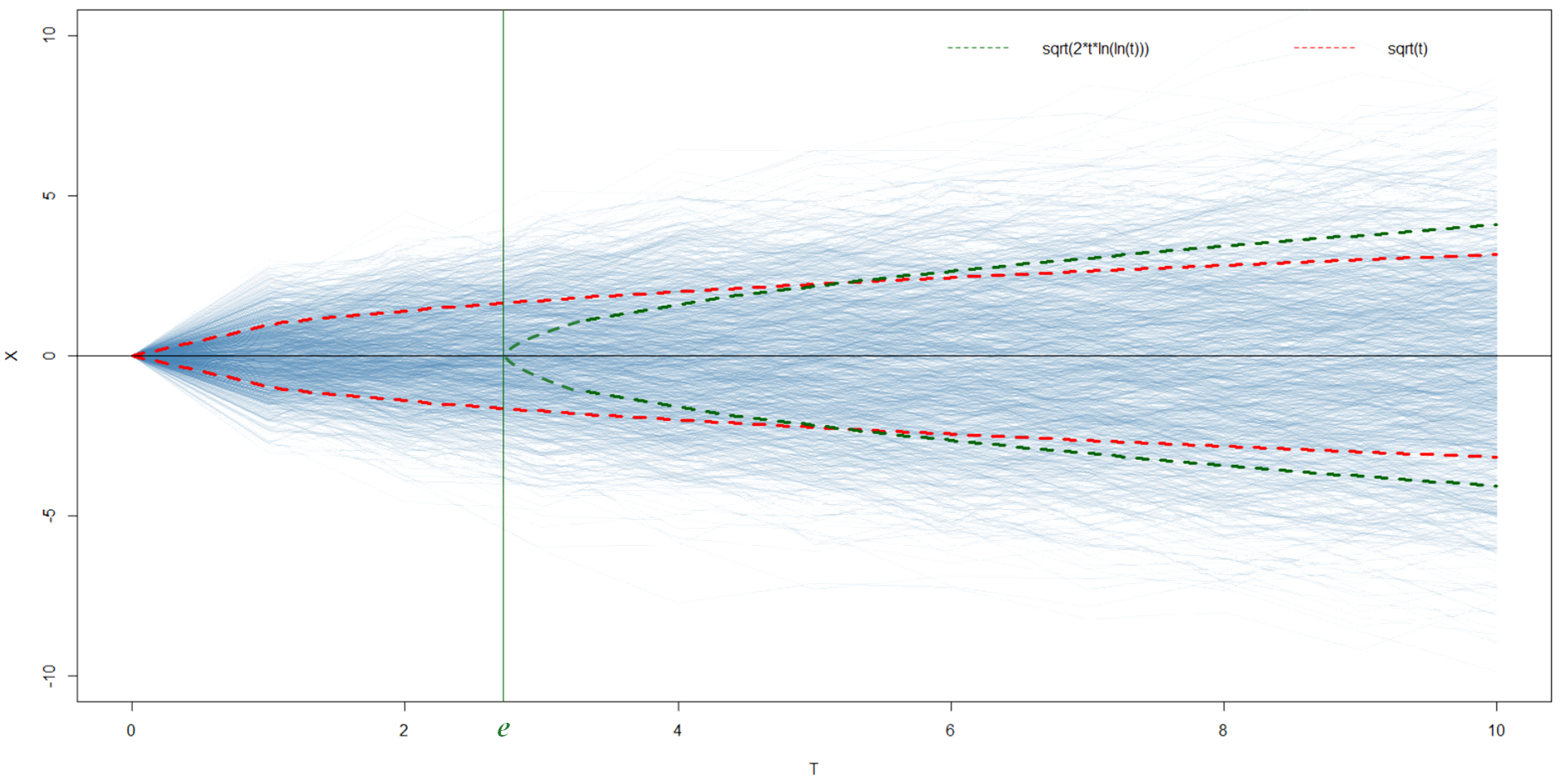}
   \caption{Law of Iterated Logarithm Simulated Zoomed}
   \label{Fig:LawofIteratedLogarithmSimulatedZoomed}
\textbf{\footnotesize \noindent
Green envelope = $\boldsymbol{\sqrt{2t \ln(\ln(t))}}$ $\quad$ , $\quad$ Red envelope = $\boldsymbol{\sqrt{t}}$.
}
\end{figure}
%\FloatBarrier

\bigskip \noindent
From Figure \ref{Fig:LawofIteratedLogarithmSimulatedZoomed}, the LIL can be ignored for $t \in [0,e)$ unless we make an adjustment by $e$ to result in $\sqrt{2t \ln (\ln (t))}-e$.

\bigskip \noindent
We wish to quantify the impact of BGC on the LIL for a BGC It\^{o} process but first we review the literature on this topic.
In the Methodology section, we define the SDE of BGC in (\ref{Eq:BGC}) as,

\begin{equation}%\label{Eq:BGC}
      dX = \Big( f(X,t)  -  \sgn[X, t] \lambda (X, t)  \Big) \, dt + g(X, t) \, dW_t,  \nonumber
\end{equation}

\bigskip \noindent
along with all its terms and prove the corresponding BGC theorem.
The Results \& Discussion section then adds visually beneficial simulations to further support our BGC theorem.

%%%%%%%%%%%%%%%%%%%%%%%%%%%%%%%%%%%%%%%%%%%%%
%%%%%%%%%%%%%%%%%%
\bigskip
%\newpage
\section{Literature Review}

\noindent
The original statement of the law of the iterated logarithm (LIL) is due to Khinchin (1924) \cite{Khinchine1924} and was later restated and refined by Kolmogorov (1929) \cite{KolmogorovLIL} and is stated in theorem \ref{LawofIteratedLogarithmStated}.
Since then, there has been a tremendous amount of work on the LIL for various kinds of random sequences and for stochastic processes.
The following is a small sample of notable developments.
Hartman \& Wintner (1941) \cite{HartmanWintner1941} generalized LIL to random walks with increments with zero mean and finite variance.
Strassen (1964) \cite{Strassen} studied LIL from the point of view of invariance principles.
Feller (1969) \cite{FellerLimitTheorems} examined limit theorems for the probabilities of large deviations.
Stout (1970) \cite{Stout1970} generalized the LIL to stationary ergodic martingales.
Major (1977) \cite{MajorLIL} extended Feller's approach by proving a generalized version of Kolmogorov's LIL.
De Acosta (1983) \cite{deAcosta1983} gave a simple proof of the so-called `Hartman-Wintner' version of LIL.
Wittmann (1985) \cite{Wittmann1985} generalized the `Hartman-Wintner' version of LIL to random walks satisfying milder conditions.
Vovk (1987) \cite{Vovk1987} derived a version of LIL valid for a single chaotic sequence.
More recently,  Berkes \& Borda (2018) \cite{BerkesBorda2019} proved the law of the iterated logarithm for $\sum _{k=1}^N \exp (2\pi i n_k \alpha )$ if the gaps $n_{k+1}-n_k$ are independent and identically distributed (iid) random variables.
Krebs (2020) \cite{Krebs2020} examined the LIL and related strong invariance principles for functionals in stochastic geometry.

\bigskip \noindent
Despite these deep results in the application of real analysis to random variables, there was a gap in the literature when it comes to the application of LIL to to It\^{o} diffusions until the 1950s.
Tanaka (1958) \cite{Tanaka1958} was one of the first to extend the LIL to one-dimensional diffusion processes.
Motoo (1959) \cite{Motoo1959} then provided a proof of the LIL through the use of the diffusion equation.
Mishra \& Acharya (1983) \cite{MishraAcharya1983} introduced normalization in the LIL of diffusions.
Kawazu {\it et. al.} (1989) \cite{KawazuTamuraTanaka1989} extended the limit theorems for the asymptotic extremes of diffusions.
Mao (2008) \cite{Mao2008} simplified the previous research on LIL for diffusions and defined the paremeters $\rho$ and $L$ for a proof that extends to $n$-Dimensional diffusions.
Finally, Appleby \& Wu (2009) \cite{ApplebyWu2009} provide a modern synthesis of the main research in LIL for diffusions by examining the solutions of SDEs that obey the LIL.
They also provide application of these bounds to financial markets. 
Appleby \& Appleby-Wu (2013) \cite{ApplebyApplebyWu2013} then also examined recurrent solutions these SDEs that obey the LIL.

\bigskip \noindent
Having reviewed the literature of the LIL, the Methodology section formulates the novel impact that BGC has on the unconstrained LIL.

%%%%%%%%%%%%%%%%%%%%%%%%%%%%%%%%%%%%%%%%%%%%%
%%%%%%%%%%%%%%%%%%
\bigskip
%\newpage
\section{Methodology}

\bigskip \noindent
\begin{theorem}\label{LawofIteratedLogarithmStated}
\textbf{(Kolmogorov's Law of Iterated Logarithm (LIL) \cite{KolmogorovLIL})}.
Let $X_1, X_2, \dots$ be independent and identically distributed (iid) random variables where $\mathbb{E}(X_i) = 0$, $\mathbb{E}(X^{2}_{i}) = \sigma^{2}_{i}$, $\forall i \in \mathbb{N}$.
Define $S_n = \sum^{n}_{i=1} X_i$, $B_n = \sum^{n}_{i=1} \sigma^{2}_{i}$, $\forall i,n \in \mathbb{N}$ and let $B_n \rightarrow \infty$.
Assume the existence of a numerical sequence $M_n$, $\forall i \in \mathbb{N}$ such that,

\[
  M_n = \mathcal{O} \Bigg( \sqrt{ \frac{B_n}{\ln(\ln(B_n))} } \Bigg)
\quad , \quad
 \mathbb{P}\Big(|X_n| \leq M_n \Big) = 1.
\]

\bigskip \noindent
Then the following relation is true,

\[
 \mathbb{P} \Bigg( \limsup\limits_{n \rightarrow \infty} \Bigg\{ \frac{S_n}{\sqrt{2B_n \ln(\ln(B_n))}} \Bigg\} = 1 \Bigg) = 1.
\]  
\hfill    $\blacksquare$
\end{theorem}

\bigskip \noindent
The interested reader is invited to see the proof of this theorem, in its original Russian wording (Kolmogorov, 1929 \cite{KolmogorovLIL}).
The denominator of this infinite sequence is the envelope bounds as seen in Figure \ref{Fig:LawofIteratedLogarithmSimulated}, whereas the numerator is of the random sequence $S_n$ itself.
The fact that the ratio approaches unity indicates that $S_n$ is bounded above and below by the envelope. 
Such sequences are useful for Markov chains, martingales, L\'{e}vy processes and other theoretical topics regarding the asymptotic nature of random series.
Instead, we wish to focus on the LIL for It\^{o} processes, especially under BGC.
We also use the natural logarithm $\ln(x)$ notation rather than the traditional $\log(x)$ used in the literature because $\log(x)=\log_{\kappa} (x)$ is usually reserved for $\kappa$=2 or 10, not for $\kappa=e$, which is what the literature is referring to.

%%%%%%%%%%%%%%%%%%%%%%%%%%%%%%%%%%%%%%%%%%%%%
%%%%%%%%%%%%%%%%%%%%%%%%%%%%%%%%%%%%%%%%%%%%%
%%%%%%%%%%%%%%%%%%%%%%%%%%%%%%%%%%%%%%%%%%%%%
%%%%%%%%%%%%%%%%%%%%%%%%%%%%%%%%%%%%%%%%%%%%%
\bigskip \noindent
Standard Brownian Motion (SBM) is defined as $dX = f (X, t) \, dt + g (X, t)  \, dB_t$, and Geometric Brownian Motion (GBM) is defined as $dX = f (X) \, dt + g (X)  \, dB_t$, where $B_t$ is a Brownian process.
However, Brownian motion does not have a drift term $f(X, t)$ but only a diffusion term $g(X, t)$ and we prefer to use the more mathematically precise Wiener process $W_t$ definition instead of $B_t$.
We thus define the following mutually exclusive terminology, in decreasing order of complexity.

\bigskip \noindent
\textbf{Standard It\^{o} Diffusion:}
\begin{equation}\label{Eq3_4}
dX = f (X, t) \, dt + g (X, t)  \, dW_t, \quad t \geq 0,
\end{equation}

\bigskip \noindent
\textbf{Geometric It\^{o} Diffusion:}
\begin{equation}\label{Eq3_2}
dX = f (X) \, dt + g (X)  \, dW_t, \quad t \geq 0,
\end{equation}

\bigskip \noindent
\textbf{Geometric Brownian Motion / Geometric Wiener Process:}
\begin{equation}\label{Eq3_1}
dX =  g \, dW_t, \quad t \geq 0, \quad \text{i.e. $g (X, t) \Rightarrow g \in \mathbb{R}$,}
\end{equation}

\bigskip \noindent
\textbf{Standard Brownian Motion / (Standard) Wiener Process:}
\begin{equation}\label{Eq3_3}
dX = dW_t, \quad t \geq 0, \quad \text{i.e. $g (X, t) \Rightarrow 1$,}
\end{equation}

\bigskip \noindent
where, $f(x): \mathbb{R} \rightarrow \mathbb{R}$, $g(x): \mathbb{R} \rightarrow \mathbb{R}$, $\forall x \in \mathbb{R}$ and,

\bigskip 
\begin{equation}\label{limits}
   \lim_{x \rightarrow \infty} f(x) \rightarrow \mu \quad , \quad \lim_{x \rightarrow \infty} g(x) \rightarrow \sigma.
\end{equation}

\bigskip \noindent
Depending on whether the generalized $f(x)$ and $g(x)$ are used, or whether the simplified $\mu$ and $\sigma$ are used, then the resulting theorems will either have more complexity, or less complexity, respectively.
Before presenting these theorems, we list some of the relationships between $f(x)$ and $g(x)$ that are examined in Appleby \& Wu (2009), 

\bigskip
\begin{equation}
f(x) = \mathcal{O} \Big(g^{-1}(x) \Big)  \Longleftrightarrow \limsup_{t \rightarrow \infty} \big| f(t) \big| \big| g(t) \big| < \infty,
\end{equation}

\bigskip \noindent
and $f(x)$ and $g(x)$ are further related by the scale function $s_c (x)$ and the speed measure $m(dx)$ of the SDE defined by,

\bigskip
\begin{equation}\label{Eq2_1}
 \resizebox{0.9\hsize}{!}{$
s_c (x) = \displaystyle \int^{x}_{c} \exp \Bigg( -2 \int^{y}_{c} \frac{f(z)}{g^2 (z)} \, dz \Bigg) \, dy, \quad
 m(dx) = \frac{2}{s' (x) g^2 (x)} \, dx,
$}
\end{equation}

\bigskip \noindent
$c,x \in I:=(l,r)$ respectively, where $I$ is the state space of the process.
The next theorem can now address the `geometric' It\^{o} diffusions of (\ref{Eq3_2}).

\bigskip
\begin{theorem} (Motoo, 1959). Let $X$ be the unique continuous real-valued process satisfying the autonomous SDE as defined in (\ref{Eq3_2}) with $X(0) = x_0$.
Let $s$ and $m$ be the scale function and speed measure of $X$ as defined in (\ref{Eq2_1}), and let $h : (0,\infty) \rightarrow (0,\infty)$ be an increasing function with $h(t) \rightarrow \infty$ as $t \rightarrow \infty$.
If $X$ is recurrent on $(l,\infty)$ (or $[l,\infty)$ in the case when $l$ is an instantaneous reflecting point) and $m(l,\infty) < \infty$, then,

\bigskip 
\[
  \mathbb{P} \Bigg[ \limsup_{t \rightarrow \infty} \frac{|W(t)|}{\sqrt{2t \ \ln(\ln(t))}} \geq 1 \Bigg] = 1 \text{ or } 0,
\]

\bigskip \noindent
depending on whether,

\[
 \int^{\infty}_{t_0} \frac{1}{s(h(t))} \, dt = \infty, \text{ or }  \int^{\infty}_{t_0} \frac{1}{s(h(t))} \, dt < \infty, \text{ respectively},
\]

\bigskip \noindent
for some  $t_0 >0$.  \hfill    $\blacksquare$
\end{theorem}

\bigskip \noindent
Mao (2008) extends Motoo's work by defining the terms $\rho$ and $K$ as,

\bigskip
\[
xf(x) < \rho, \quad \big| \big|  g(x, t)  \big| \big| \leq K,
\]

\bigskip \noindent
where by `$\leq$', it is meant that {\it most} $X(t)$ paths will not exceed these bounds, but some will as shown in Figure \ref{Fig:LawofIteratedLogarithmSimulated}.
Mao then proves the corresponding LIL,

\bigskip
\[
\limsup_{t \rightarrow \infty} \frac{|X(t)|}{\sqrt{2t \, \ln(\ln(t))}} = K \sqrt{e} \quad a. s.,
\] 

\bigskip \noindent
with,

\bigskip 
\[
\lim_{x \rightarrow - \infty} xf(x) = L_{- \infty} > \frac{\sigma^2}{2},
\]

\bigskip \noindent
and,

\bigskip 
\[
\lim_{x \rightarrow \infty} \frac{xf(x)}{g^2 (x)} = L_{ \infty} > \frac{1}{2}.
\]

\bigskip \noindent
Note that $L_{\infty} < \infty$ as $x \rightarrow \infty$ when $f(x)$ and $g(x)$ are regularly varying at infinity.
Mao's findings add greater understanding on the relationship between $f(x)$ and $g(x)$ but it is not an exact bound(s), and so the following theorem addresses the `geometric' It\^{o} diffusions of (\ref{Eq3_2}) more precisely than Motoo's theorem.

\bigskip
%---------------------------------------------------------------------------------------
\begin{theorem} (Appleby \& Wu, 2009, p.928).
Let $X$ be the unique continuous adapted process which obeys (\ref{Eq3_2}).
Let $\Omega := \{ \omega : \lim_{ t \rightarrow \infty} X(t,\omega) = \infty \}$.
If,

\bigskip
\begin{equation}\label{Eq:4.8}
   \lim_{x \rightarrow \infty} x f(x) = L_{\infty}, \quad g(x) = \sigma,
\end{equation}

\bigskip \noindent
where $x \in \mathbb{R}$, $\sigma \ne 0$ and $L_{\infty} > \sigma^2 / 2$, then $\mathbb{P}[\Omega] > 0$ and $X$ satisfies,

\bigskip
\begin{equation}\label{Eq:4.9}
   \limsup_{t \rightarrow \infty} \frac{|X(t)|}{\sqrt{2t \, \ln(\ln(t))}} = |\sigma|  \quad a. s. \text{ on } \Omega,
\end{equation}

\bigskip \noindent
and,

\bigskip
\begin{equation}\label{Eq:4.10}
   \liminf_{t \rightarrow \infty} \frac{\ln \Big( \frac{X(t)}{\sqrt{t}} \Big)}{\ln(\ln(t))} = - \frac{1}{\frac{2 L_{\infty}}{\sigma^2}-1}  \quad a. s. \text{ on } \Omega.
\end{equation} \hfill    $\blacksquare$
\end{theorem}

\bigskip \noindent
Finally, the following theorem by the same authors extends their previous theorem to `standard' It\^{o} diffusions of (\ref{Eq3_4}).

\bigskip
%---------------------------------------------------------------------------------------
\begin{theorem}\label{ApplebyWu2009p932} (Appleby \& Wu, 2009, p.932).
Let $X$ be the unique continuous adapted process which obeys (\ref{Eq3_2}).
Let $\Omega := \{ \omega : \lim_{ t \rightarrow \infty} X(t,\omega) = \infty \}$.
If there exist positive real numbers $L_{\infty}$ and $\sigma$ such that $L_{\infty} > \sigma^2 /2$, $f(x)$ obeys (\ref{Eq:4.8}), and $g(x)$ obeys,

\begin{equation}\label{Eq:4.22}
\forall x \in \mathbb{R}, \quad g(x) \ne 0, \quad \lim_{x \rightarrow \infty} g(x) = \sigma \in \mathbb{R}_{+},
\end{equation}

\bigskip \noindent
then $X$ satisfies (\ref{Eq:4.9}) and (\ref{Eq:4.10}). \hfill    $\blacksquare$
\end{theorem}

\bigskip \noindent
Having reviewed all the current research that relates $f(x)$ and $g(x)$ in the most time dependent setting of $(X,t)$, we now define the SDE of BGC stochastic processes.

\bigskip \noindent
\begin{definition}\textbf{(SDE of BGC Stochastic Process)}. For a complete filtered probability space $(\Omega, \mathcal{F}, \{ \mathcal{F} \}_{t \geq 0}, \mathbb{P})$ and a BGC function $\Psi (x) : \mathbb{R} \rightarrow \mathbb{R}$, $\forall x \in \mathbb{R}$, then the corresponding BGC It\^{o} diffusion is expressed as,

\bigskip \noindent
\begin{equation}
   \begin{array}{rcl}
      dX & = & \Big( f(X,t) \, dt + g(X, t) \, dW_t \Big) - \sgn \Big[ f(X,t) \, dt + g(X, t) \, dW_t \Big] \Psi (X, t) \, dt \\
           & = & \Big( f(X,t) \, dt + g(X, t) \, dW_t \Big) - \sgn [X, t] \Psi (X, t) \, dt \\
           & = & \Big( f(X,t)  -  \sgn[X, t] \Psi (X, t)  \Big) \, dt + g(X, t) \, dW_t, \label{Eq:BGC}
   \end{array}
\end{equation}

\bigskip \noindent
where $\sgn[x]$ is defined in the usual sense as,

\bigskip \noindent
\begin{eqnarray*}
   \sgn [x] & = &
{
   \begin{cases}
\displaystyle  \phantom{-}1 &, \quad x > 0 \\
\displaystyle  \phantom{-}0 &, \quad x = 0 \\
\displaystyle                   -1 &, \quad x < 0
   \end{cases}},
\end{eqnarray*}

\bigskip \noindent
and $f(X, t)$ and $g(X, t)$ are convex functions. \hfill    $\blacksquare$
\end{definition}

\bigskip \noindent
Now, the unconstrained stochastic process $X(t)$ is bounded by,

\bigskip 
\[
- \sqrt{2t \, \ln(\ln(t))} \leq X(t) \leq \sqrt{2t \, \ln(\ln(t))}.
\]

\bigskip \noindent
At a high level, the BGC stochastic process $\overline{X(t)}$ is bounded by either,

\bigskip 
\[
- \sqrt{2t \, \ln(\ln(t))} + \Phi(t) \leq \overline{X(t)} \leq \sqrt{2t \, \ln(\ln(t))} -  \Phi(t),
\]

\bigskip \noindent
for some function $\Phi(t)$ or by,

\bigskip 
\[
- \Gamma (t) \leq \overline{X(t)} \leq \Gamma (t), 
\]

\bigskip \noindent
for some other function $\Gamma (t)$ where,

\bigskip 
\[
 \big| \Gamma (t) \big|  \leq \big| \sqrt{2t \, \ln(\ln(t))}  +  \Phi(t) \big|.
\]

\bigskip \noindent
We will thus need to derive either $\Phi(t)$ or $\Gamma (t)$.
The simplest possible theorem for BGC is to derive a new theorem from theorem \ref{ApplebyWu2009p932}.

\bigskip
\begin{theorem}
\textbf{(BGC Mapping for LIL)}.
Let $X(t)$ be the unique continuous adapted process which obeys (\ref{Eq3_2}) and $\overline{X(t)}$ be the corresponding BGC process.
If there exist positive real numbers $L_{\infty}$ and $\sigma$ such that $L_{\infty} > \sigma^2 /2$, $f(x)$ obeys (\ref{Eq:4.8}) and $g(x)$ obeys (\ref{Eq:4.22}), then defining $F(x)$ and $G(x)$ as BGC versions of $f(x)$ and $g(x)$ as,

\begin{equation}\label{Corollary_Aldo_1}
F(x) := f(x) - \sgn[x]\Psi(x), \quad G(x) := g(x) - \sgn[x]\Psi(x),
\end{equation}

\bigskip \noindent
where,

\begin{equation}\label{Corollary_Aldo_2}
F(x): \mathbb{R} \rightarrow \mathbb{R}, \quad G(x): \mathbb{R} \rightarrow \mathbb{R}, \quad \forall x \in \mathbb{R},
\end{equation}

\bigskip \noindent
implies that $\overline{X(t)}$ also satisfies (\ref{Eq:4.9}) and (\ref{Eq:4.10}).
\end{theorem}

\begin{proof}
\noindent
In general, for linear combinations of drift functions $f_i (x)$ and diffusion functions $g_j (x)$ respectively, $\forall a_i, b_j, n, m \in \mathbb{N}$, then,

\[
\begin{array}{lcr}
a_1 f_1 (x) + \cdot \cdot \cdot + a_n f_n (x)  & = &  f_{n+1} (x) \\
b_1 g_1 (x) + \cdot \cdot \cdot + b_m g_m (x) & = &  g_{m+1} (x)
\end{array},
\]

\bigskip \noindent
are also drift and diffusion functions respectively.
However, one can not simply end here since, after some minor abuse of notation, at sucessive time instances,

\bigskip \noindent
\begin{equation}
   \begin{array}{rcl}
      X_1 & = & \Big( f(X_0)  -  \underbrace{\sgn[X_0]}_{\text{Impacts $X_1$}} \underbrace{\Psi (X_0)}_{\text{Impacts $X_1$}}  \Big) \, dt + g(X_0) \, dW_0, \nonumber \\
      X_2 & = & \Big( \underbrace{f(X_1)  -  \sgn[X_1] \Psi (X_1)}_{\text{Impacts $X_2$}}  \Big) \, dt + \underbrace{g(X_1)}_{\text{Impacts $X_2$}} \, dW_1, \nonumber \\
\boldsymbol{\cdot \cdot \cdot}     & \boldsymbol{\cdot \cdot \cdot} & \boldsymbol{\cdot \cdot \cdot} \nonumber
   \end{array}
\end{equation}

\bigskip \noindent
shows that where BGC initially is acting on the drift at $t=0$, it then impacts the diffusion (and drift) at $t=1$, and so on, due the the iterative nature of It\^{o} diffusions.
This means our proof must not only consider the impact of BGC on $f(x)$ but also on $g(x)$.
Now, (\ref{Corollary_Aldo_1}) is true for the trivial case if $\sgn[x]=1$ (i.e. there is no sign oscillation or sign switching).
In all other cases of $\sgn[x]$, as the BGC It\^{o} process $\overline{X(t)}$ oscillates above and below the $X = 0$ axis, it simply requires us to state that,

\begin{equation}\label{Corollary_Aldo_3}
%\[
F(x): \mathbb{R}_{-} \cup \mathbb{R}_{+} \rightarrow \mathbb{R}, \quad 
G(x): \mathbb{R}_{-} \cup \mathbb{R}_{+} \rightarrow \mathbb{R}.
\end{equation}
%\]

\bigskip \noindent
The only step remaining is to show that $F(x)$ and $G(x)$ guarantee that there exists a unique global continuous solution to the BGC SDE by verifying that these functions satisfy Lipschitz local continuity
% and Linear growth boundedness
.
The Lipschitz local continuity condition is satisfied if there exists a positive real constant $\Lambda_1$ such that $\forall x_1, x_2 \in \mathbb{R}$,

\begin{equation*}
\begin{array}{rcl}
\displaystyle \big|F(x_{1}, t)-F(x_{2}, t) \big| & \leq & \Lambda_1 \big| x_{1}-x_{2} \big|, \quad  t \geq 0, \\ \\
\displaystyle \big|G(x_{1}, t)-G(x_{2}, t) \big| & \leq & \Lambda_1 \big| x_{1}-x_{2} \big|, \quad t \geq 0.
\end{array}
\end{equation*}

\bigskip \noindent
Substituting (\ref{Corollary_Aldo_3}) for $f(x)$ and $g(x)$ into theorem \ref{ApplebyWu2009p932} completes the proof.
\end{proof}

\bigskip \noindent
For $F(x)$ and $G(x)$ to guarantee that there exists a unique {\it \textbf{strong}} global continuous solution to the BGC SDE, we would also need to establish that the Linear growth bound condition is satisfied, which is met if there exists a positive real constant $\Lambda_2$ such that $\forall x \in \mathbb{R}$,

\begin{equation}\label{lineargrowthbound}
\displaystyle \big|F(x, t) \big|^2 + \big| G(x, t) \big|^2 \leq \Lambda_2 \Big( 1 + \big| x \big|^2 \Big), \quad t \geq 0,
\end{equation}

\bigskip \noindent
which is dependent on the actual nature of $\Psi (X,t)$.
If $\Psi (X,t)$ does not satisfy this condition, then the solution might explode in finite time.

%%%%%%%%%%%%%%%%%%%%%%%%%%%%%%%%%%%%%%%%%%%%%
%%%%%%%%%%%%%%%%%%%%%%%%%%%%%%%%%%%%%%%%%%%%%
%%%%%%%%%%%%%%%%%%%%%%%%%%%%%%%%%%%%%%%%%%%%%
%%%%%%%%%%%%%%%%%%%%%%%%%%%%%%%%%%%%%%%%%%%%%
\bigskip
%\newpage
\section{Results and Discussion}

\noindent
To compliment the Methodology section, $\Psi (X, t)$ in (\ref{Eq:BGC}) was instantiated in this Results section to be $\Psi (X, t) := \frac{x^2}{\beta}$, where $\beta = 100$ and $x \in \mathbb{R}$, as the simplest form of BGC stochastic processes, so that a geometric context can be provided.
(\ref{Eq:BGC}) was simulated over 1,000 time steps in Figure \ref{Fig:1000Simulationsof1000Step1DimensionalItoDiffusionsWithWithoutBGC} for with and without BGC for various values of $\mu$ and $\sigma$.
We note that whilst $\mu$ and $\sigma$ are used here instead of the more general $f(X,t)$ and $g(X,t)$, we have actually catered for this by using $\Psi(X,t)$ rather than the autonomous constant $\Psi$.

\bigskip \noindent
From Figure \ref{Fig:1000Simulationsof1000Step1DimensionalItoDiffusionsWithWithoutBGC}, we can see in four different scenarios how BGC constrains the blue unconstrained It\^{o} process into the constrained red It\^{o} process, pulling it back the further it deviates from the origin in either extreme of $X$.

\bigskip \noindent
To extend this analysis further, (\ref{Eq:BGC}) was simulated over 1,000 time steps for a total of 1,000 paths and is shown in Figure \ref{Fig:DiscretizationDetailsofSimulationsduetoBGCwithnoDiffusion} with positive, zero, negative drift and with varying diffusion values.

\bigskip \noindent
From Figure \ref{Fig:DiscretizationDetailsofSimulationsduetoBGCwithnoDiffusion}, we see that most noticably in (d), that there are `gaps'.
To show this discretization in further detail, we reduced the marker size, not included the lines that connect the markers and increased their transparency.
One can also see that the discretization polarizes the markers along the direction of the drift.
To elaborate on this hidden barrier discretization phenomenon further, the densities of the simulations, before and after BGC were plotted in Figure \ref{Fig:ImpactofBGContheDistributionofItoDiffusionswithnoDiffusionTerm}.

\begin{figure}
   \centering
       \includegraphics[scale=0.63]{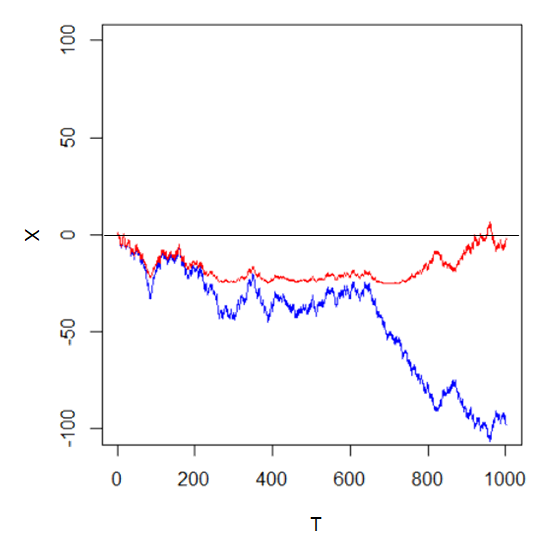}
      \includegraphics[scale=0.63]{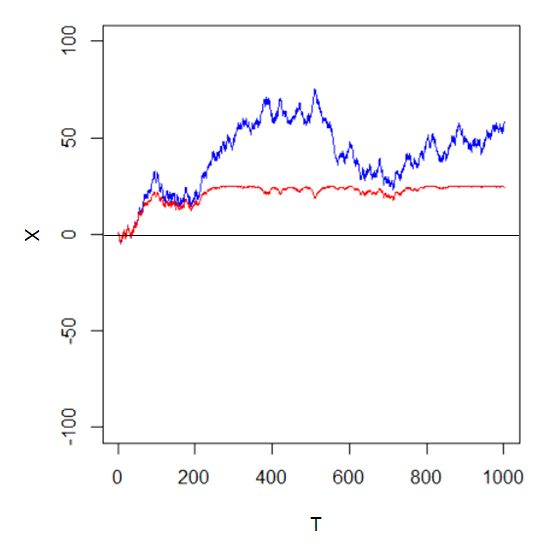}\\
\textbf{\footnotesize \noindent
(a). $\boldsymbol{\mu = -0.05, \sigma = 1}$  $\quad \quad \quad \quad \quad \quad \quad \quad \quad$ (b). $\boldsymbol{\mu = 0.05, \sigma = 1}$
}\\
      \includegraphics[scale=0.63]{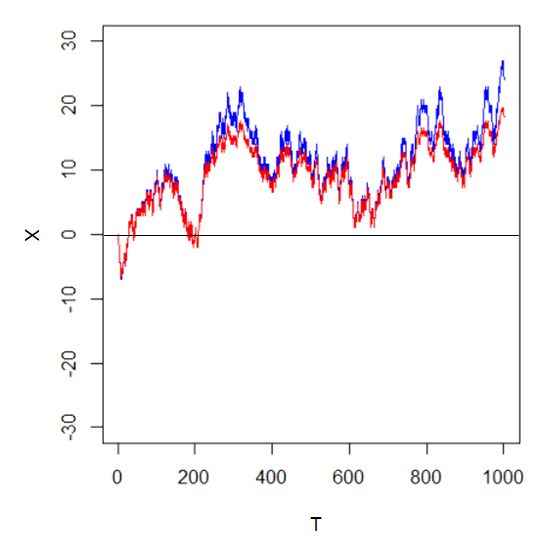}
      \includegraphics[scale=0.63]{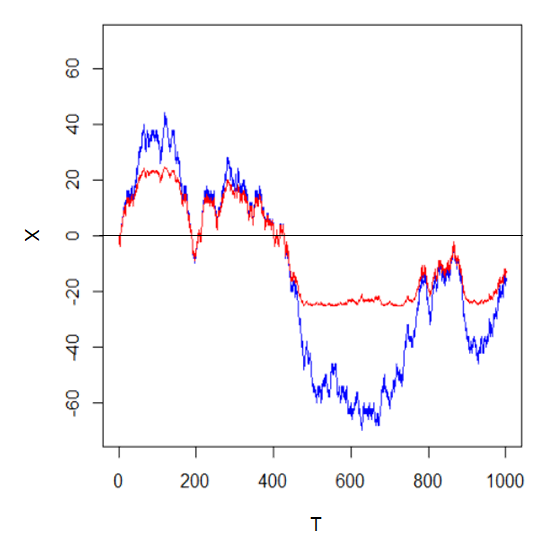}\\
\textbf{\footnotesize \noindent
(c). $\boldsymbol{\mu = 0, \sigma = 1}$  $\quad \quad \quad \quad \quad \quad \quad \quad \quad \quad$  (d). $\boldsymbol{\mu = 0, \sigma = 2}$
}\\
   \caption{4 Simulations of 1,000-Step 1-Dimensional It\^{o} Diffusions, With \& Without BGC}
   \label{Fig:1000Simulationsof1000Step1DimensionalItoDiffusionsWithWithoutBGC}
\textbf{\footnotesize \noindent
Blue = Without BGC $\quad$, $\quad$ Red = With BGC}\\
\flushleft
\textbf{\footnotesize \noindent
(a). {\it \textbf{Negative}} $\boldsymbol{\mu}$ shows most constraining occurs the {\it \textbf{lower}} $\boldsymbol{X_t}$ moves from origin.\\
(b). {\it \textbf{Positive}} $\boldsymbol{\mu}$ shows most constraining occurs the {\it \textbf{higher}} $\boldsymbol{X_t}$ moves from origin.\\
(c). {\it \textbf{Small}} $\boldsymbol{\sigma}$ shows most constraining occurs the further $\boldsymbol{X_t}$ moves from origin.\\
(d). {\it \textbf{Larger}} $\boldsymbol{\sigma}$ shows most constraining occurs the further $\boldsymbol{X_t}$ moves along $\boldsymbol{X}$.\\
}
\end{figure}
%\FloatBarrier

\bigskip \noindent
From Figure \ref{Fig:ImpactofBGContheDistributionofItoDiffusionswithnoDiffusionTerm}, we see that the hidden barrier becomes obvious due to the vertical peak(s) suddenly dropping to zero.
We also see some ``sinusoidal'' accumulations, especially in (a) and (b), which correlate with the discrete bands in Figure \ref{Fig:DiscretizationDetailsofSimulationsduetoBGCwithnoDiffusion}.
Even if the minima of these peaks are non-zero, due to the transparency feature of Figure \ref{Fig:DiscretizationDetailsofSimulationsduetoBGCwithnoDiffusion}, the relative colouring can imply that BGC causes regions where there are no simulation path values, even though some values that are represented or detected.

%------MU---------------------------------------------------------------1
\begin{figure}
   \centering
   \includegraphics[scale=0.43]{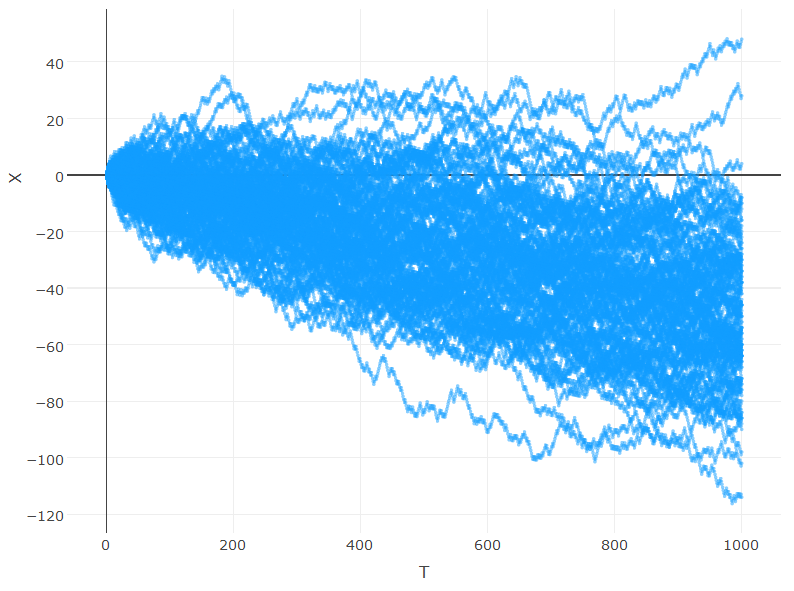}
   \includegraphics[scale=0.43]{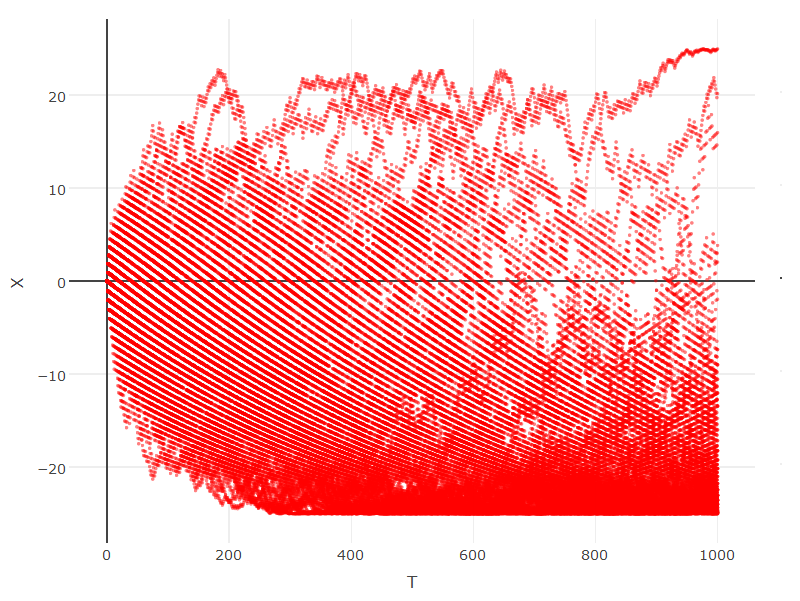}\\
\textbf{\footnotesize \noindent
(a). $\boldsymbol{\mu = -0.05,  \sigma = 1}$ without BGC \quad \quad \quad \quad \quad \quad (b). $\boldsymbol{\mu = -0.05,  \sigma = 1}$ with BGC
}\\
  \includegraphics[scale=0.43]{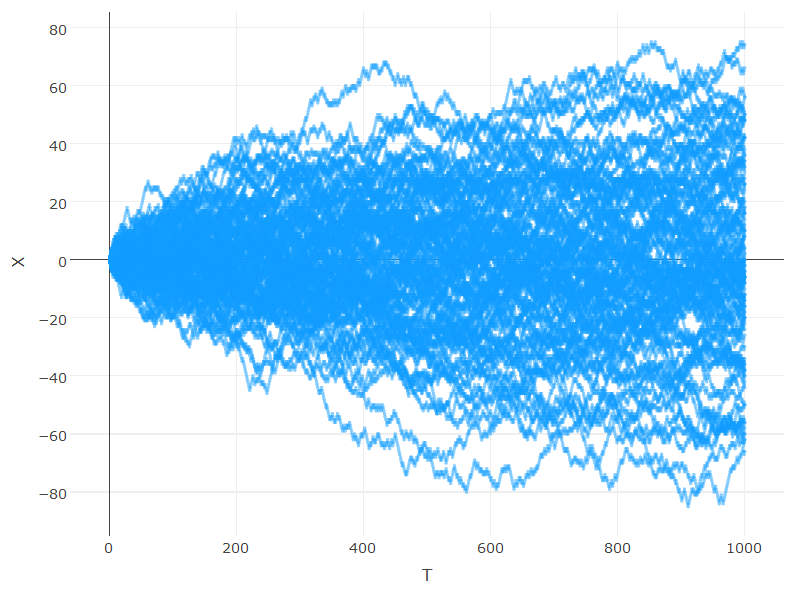}
   \includegraphics[scale=0.43]{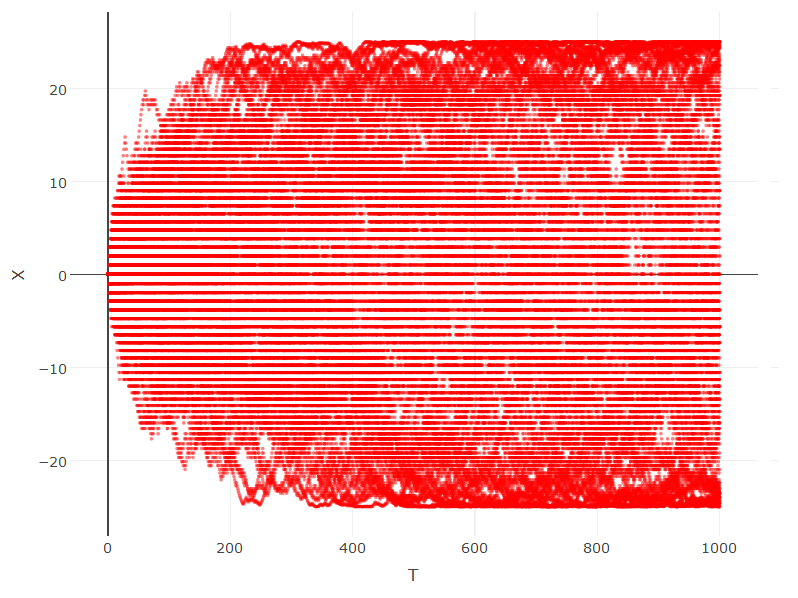}\\
\textbf{\footnotesize \noindent
(c). $\boldsymbol{\mu = 0,  \sigma = 1}$ without BGC \quad \quad \quad \quad \quad \quad (d). $\boldsymbol{\mu = 0, \sigma = 1}$ with BGC
}\\
    \includegraphics[scale=0.43]{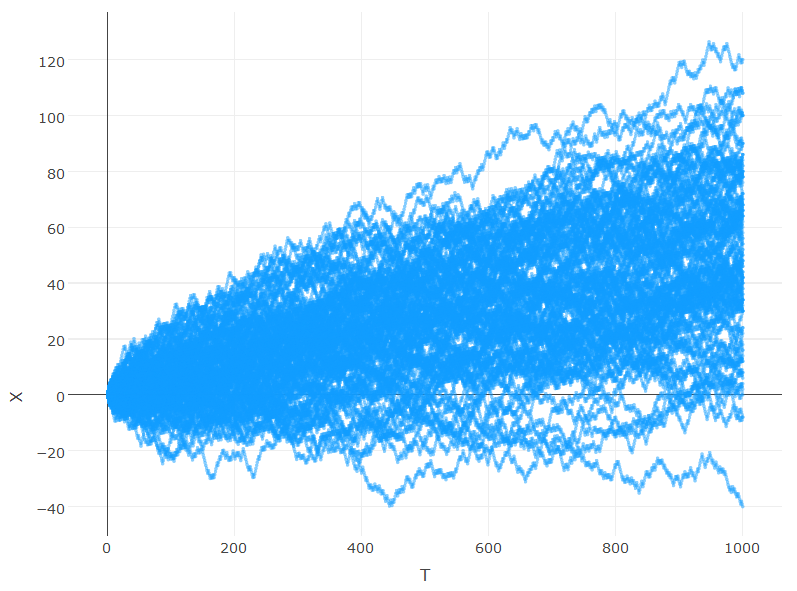}
   \includegraphics[scale=0.43]{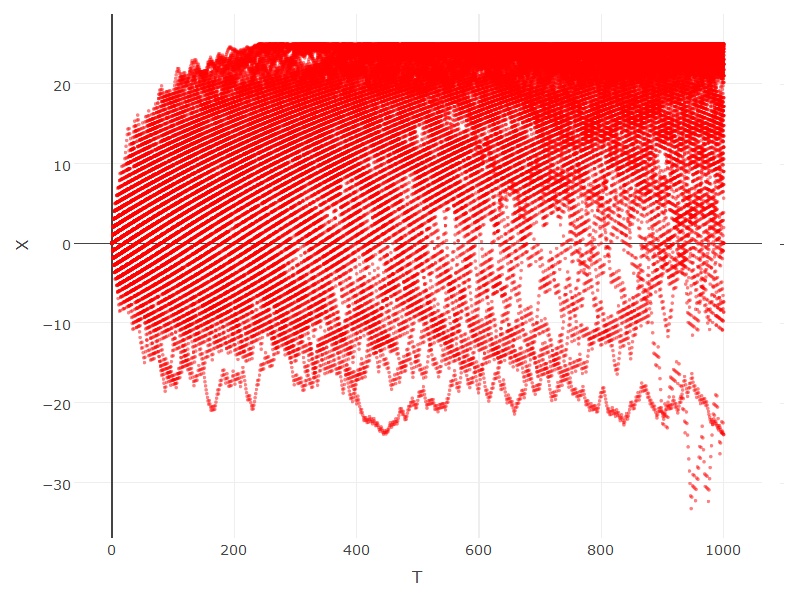}\\
\textbf{\footnotesize \noindent
(e). $\boldsymbol{\mu = 0.05,  \sigma = 1}$ without BGC \quad \quad \quad \quad \quad \quad (f). $\boldsymbol{\mu = 0.05,  \sigma = 1}$ with BGC
}\\
   \caption{Discretization Details of Simulations due to BGC with no Diffusion}
   \label{Fig:DiscretizationDetailsofSimulationsduetoBGCwithnoDiffusion}
\flushleft
\textbf{\footnotesize \noindent
(a). {\it \textbf{Negative}} drift is constrained in (b) the more it deviates away from the origin, causing {\it \textbf{downward}} diagonal bands to form.\\
(c). {\it \textbf{Zero}} drift is constrained in (d) the more it deviates away from the origin, causing horizontal bands to form.\\
(e). {\it \textbf{Positive}} drift is constrained in (f) the more it deviates away from the origin, causing {\it \textbf{upward}} diagonal bands to form.\\
(b), (d) \& (f): Exponentially `attracted' to hidden horizontal (reflective) barrier(s).}
\end{figure}
%\FloatBarrier

%------MU---------------------------------------------------------------2
\begin{figure}
   \centering
   \includegraphics[scale=0.6]{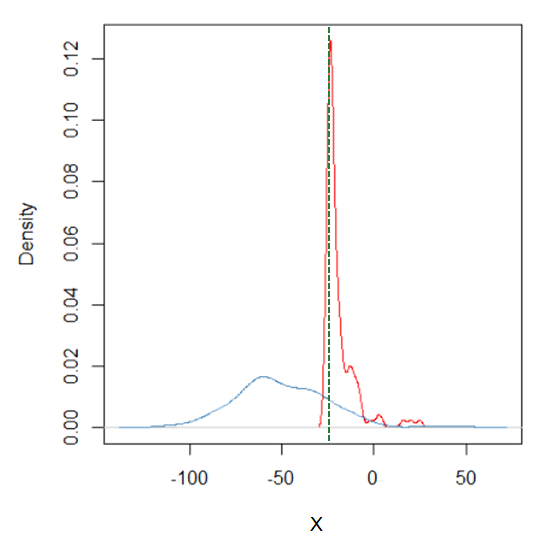}
   \includegraphics[scale=0.6]{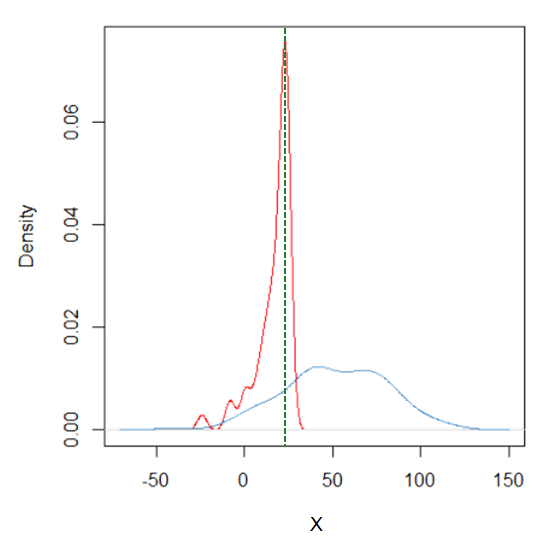}\\
\textbf{\footnotesize \noindent
(a). $\boldsymbol{\mu = -0.05, \sigma= 1}$  \quad \quad \quad \quad \quad \quad \quad \quad \quad (b). $\boldsymbol{\mu = 0.05, \sigma= 1}$}\\
   \includegraphics[scale=0.6]{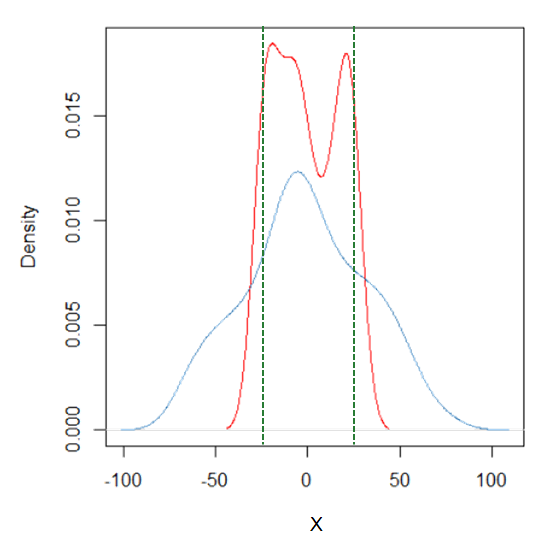}\\
\textbf{\footnotesize \noindent
(c). $\boldsymbol{\mu = 0, \sigma= 1}$}\\
   \caption{Impact of BGC on the Distribution of It\^{o} Diffusions with no Diffusion Term}
   \label{Fig:ImpactofBGContheDistributionofItoDiffusionswithnoDiffusionTerm}
\textbf{\footnotesize \noindent
Blue = original density $\quad$ , $\quad$ Red = BGC density $\quad$ , $\quad$ Green = peak.}\\
\flushleft
\textbf{\footnotesize \noindent
(a). BGC squeezes the {\it \textbf{negative}} skew distribution to {\it \textbf{positive}} direction due to impact of hidden BGC barrier.\\
(b). BGC squeezes the {\it \textbf{positive}} skew distribution to the {\it \textbf{negative}} direction due to impact of hidden BGC barrier.\\
(c). BGC squeezes the {\it \textbf{zero}} skew distribution to both the {\it \textbf{negative}} and {\it \textbf{positive}} direction due to impact of hidden BGC barrier.\\
Notice the `sinusoidal' nature of BGC distributions, which was being exhibited in Figure \ref{Fig:DiscretizationDetailsofSimulationsduetoBGCwithnoDiffusion} as banding or discretization.\\
$\quad$\\
$\quad$\\
$\quad$\\
$\quad$\\
$\quad$\\
}
\end{figure}
%\FloatBarrier

\bigskip \noindent
Having held mainly $\sigma$ insignificant by setting $\sigma = 1$ whilst varying $\mu$, we now hold $\mu$ to be insignificant by setting $\mu = 0$ whilst varying $\sigma$, as shown in Figure \ref{Fig:DiscretizationDetailsofSimulationsduetoBGCwithnoDrift}.

%------SIGMA----------------------------------------------------------1
\begin{figure}
   \centering
   \includegraphics[scale=0.25]{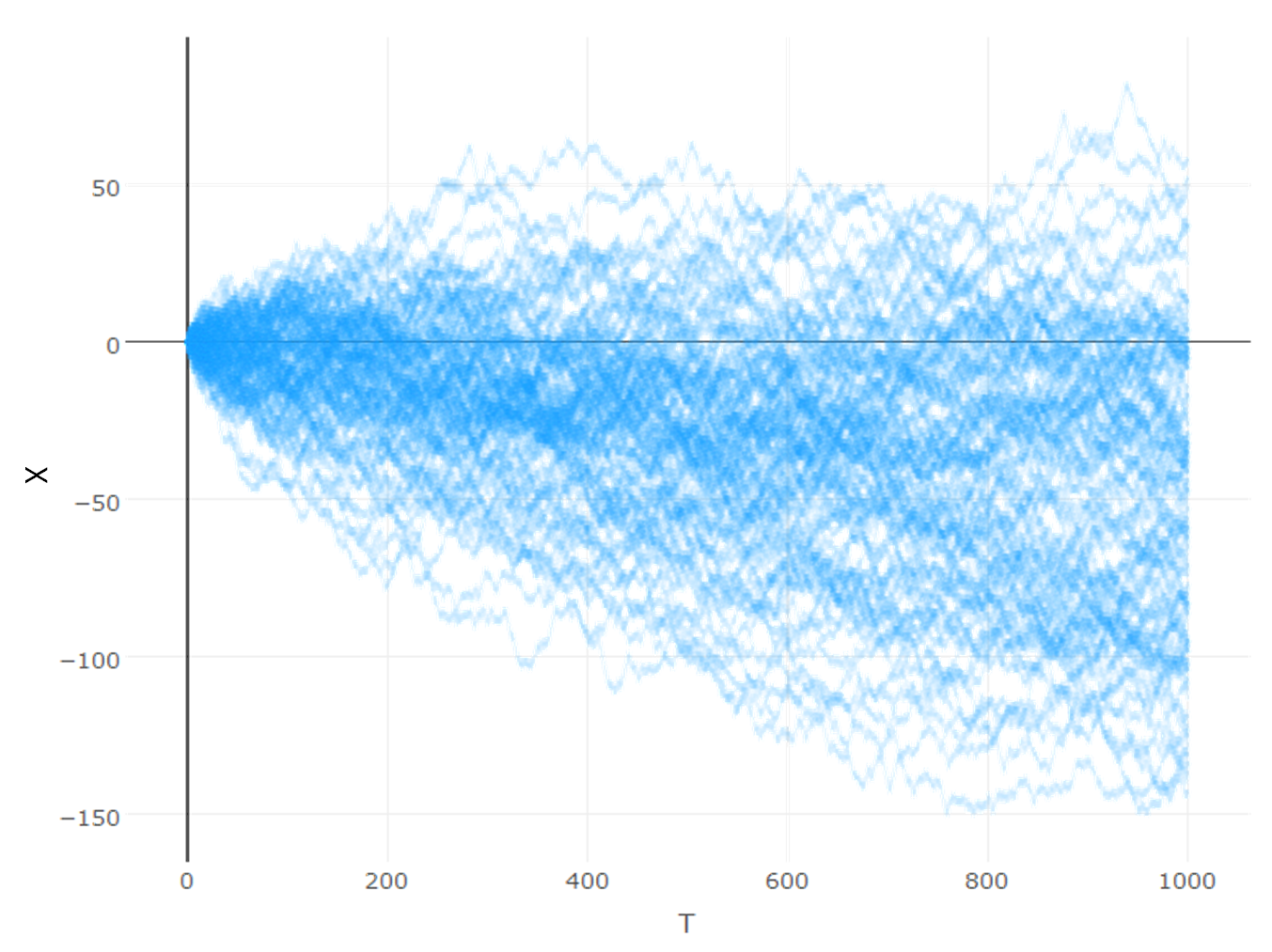}
   \includegraphics[scale=0.25]{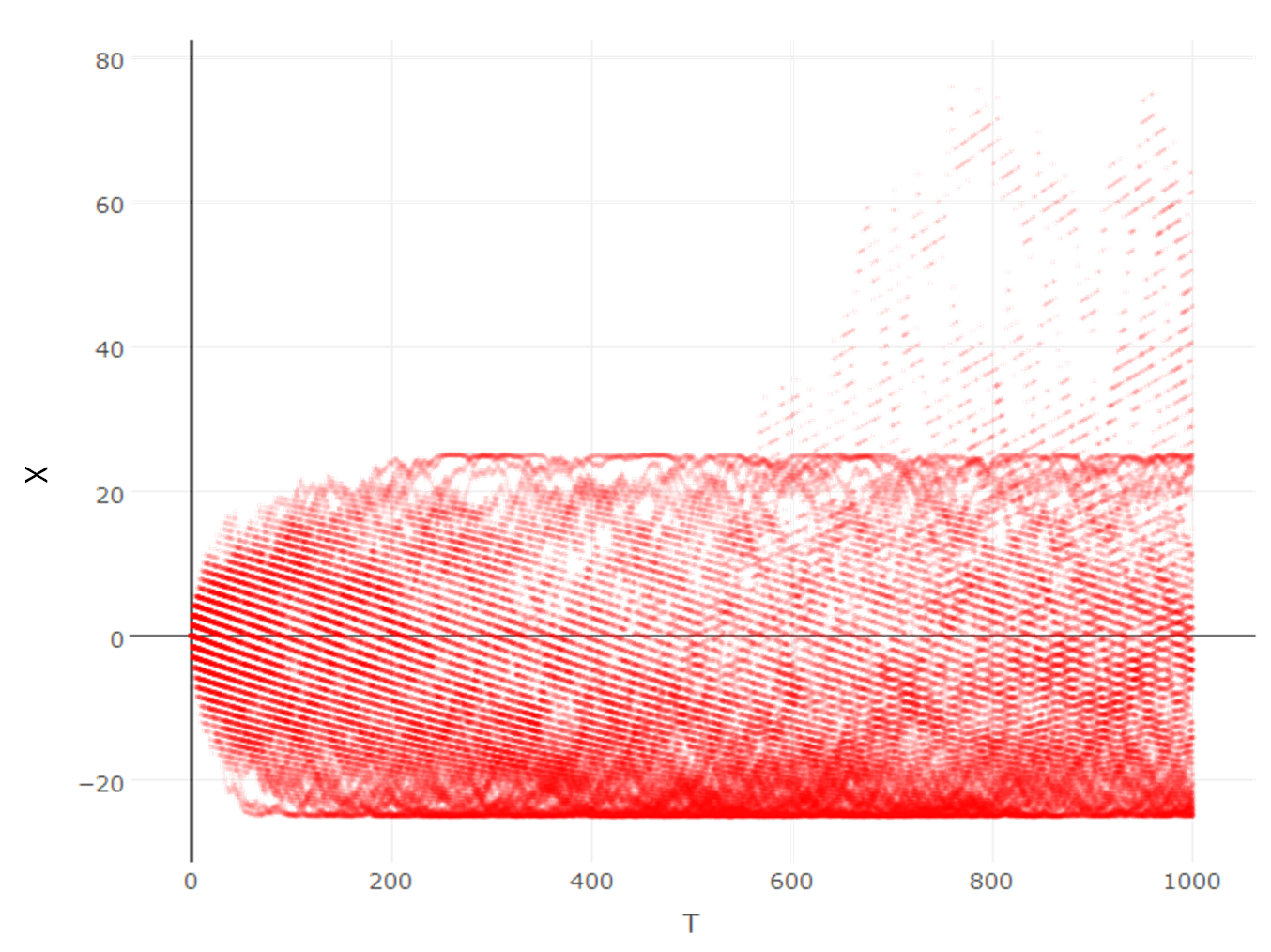}\\
\textbf{\footnotesize \noindent
(a). $\boldsymbol{\mu = 0, \sigma =  -1.5}$ without BGC \quad \quad \quad \quad (b). $\boldsymbol{\mu = 0, \sigma =   -1.5}$ with BGC
}\\
   \includegraphics[scale=0.25]{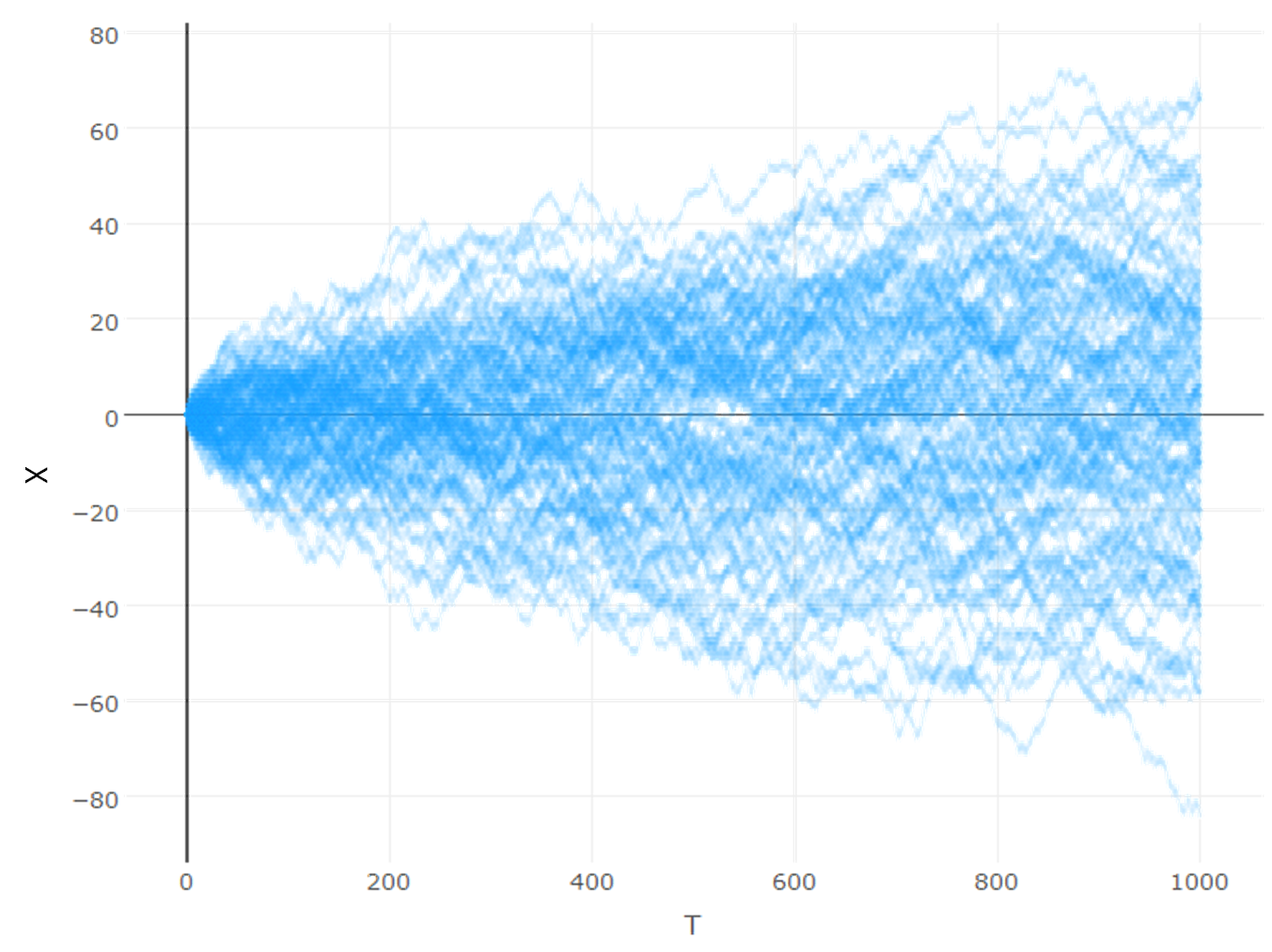}
   \includegraphics[scale=0.25]{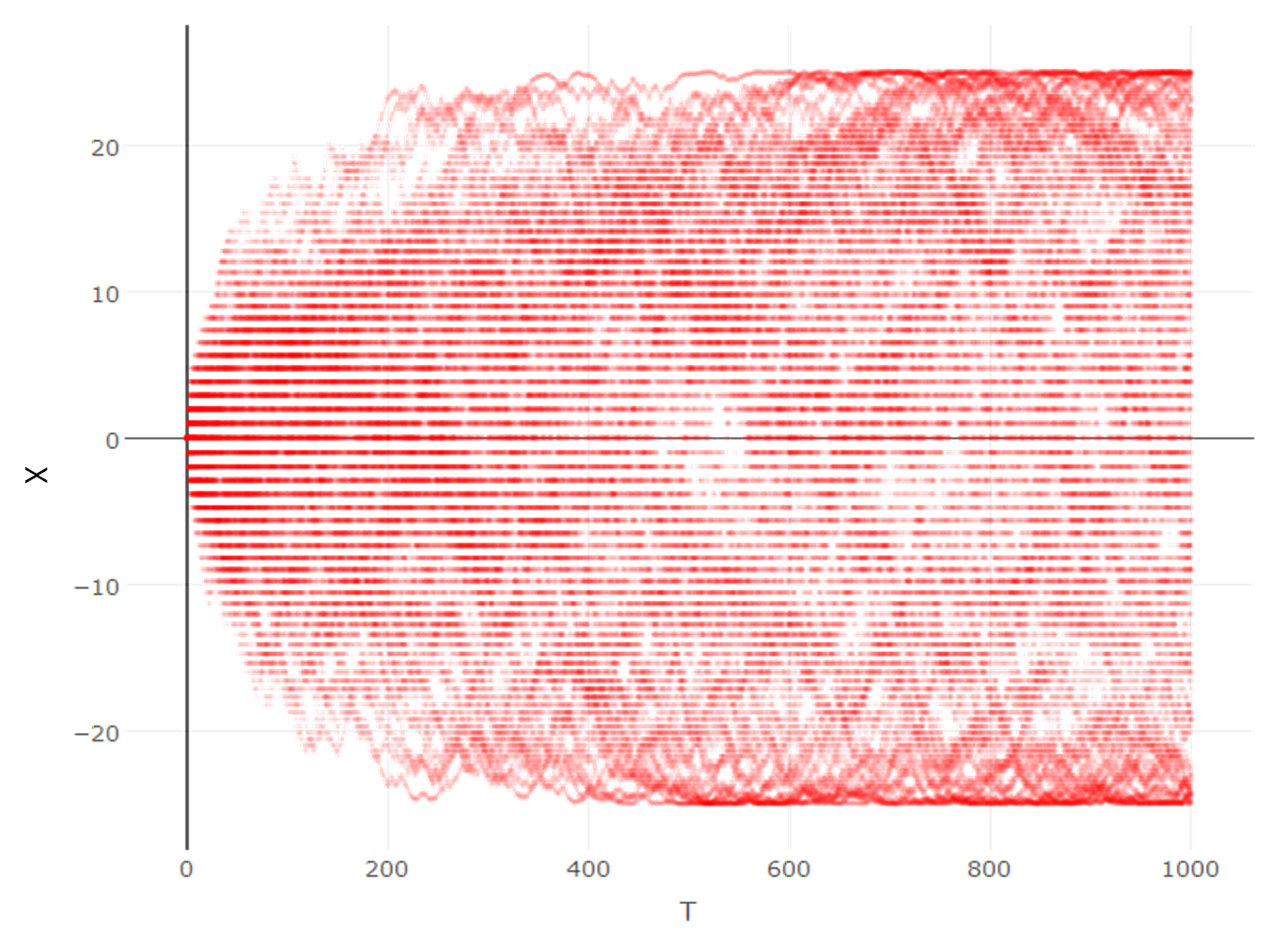}\\
\textbf{\footnotesize \noindent
(c). $\boldsymbol{\mu = 0, \sigma =  1 }$ without BGC \quad \quad \quad \quad (d). $\boldsymbol{\mu = 0, \sigma =  1 }$ with BGC
}\\
   \includegraphics[scale=0.25]{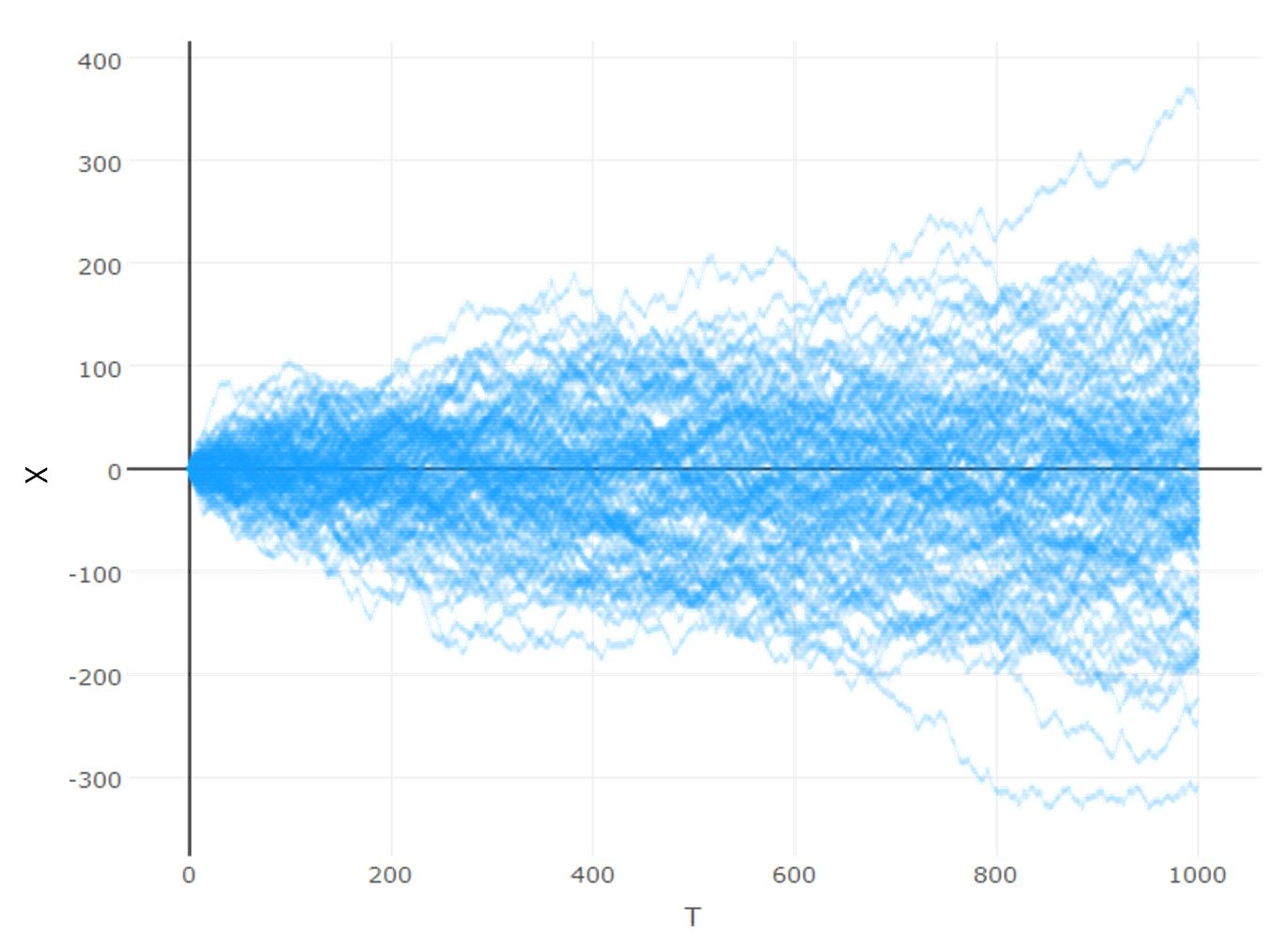}
   \includegraphics[scale=0.25]{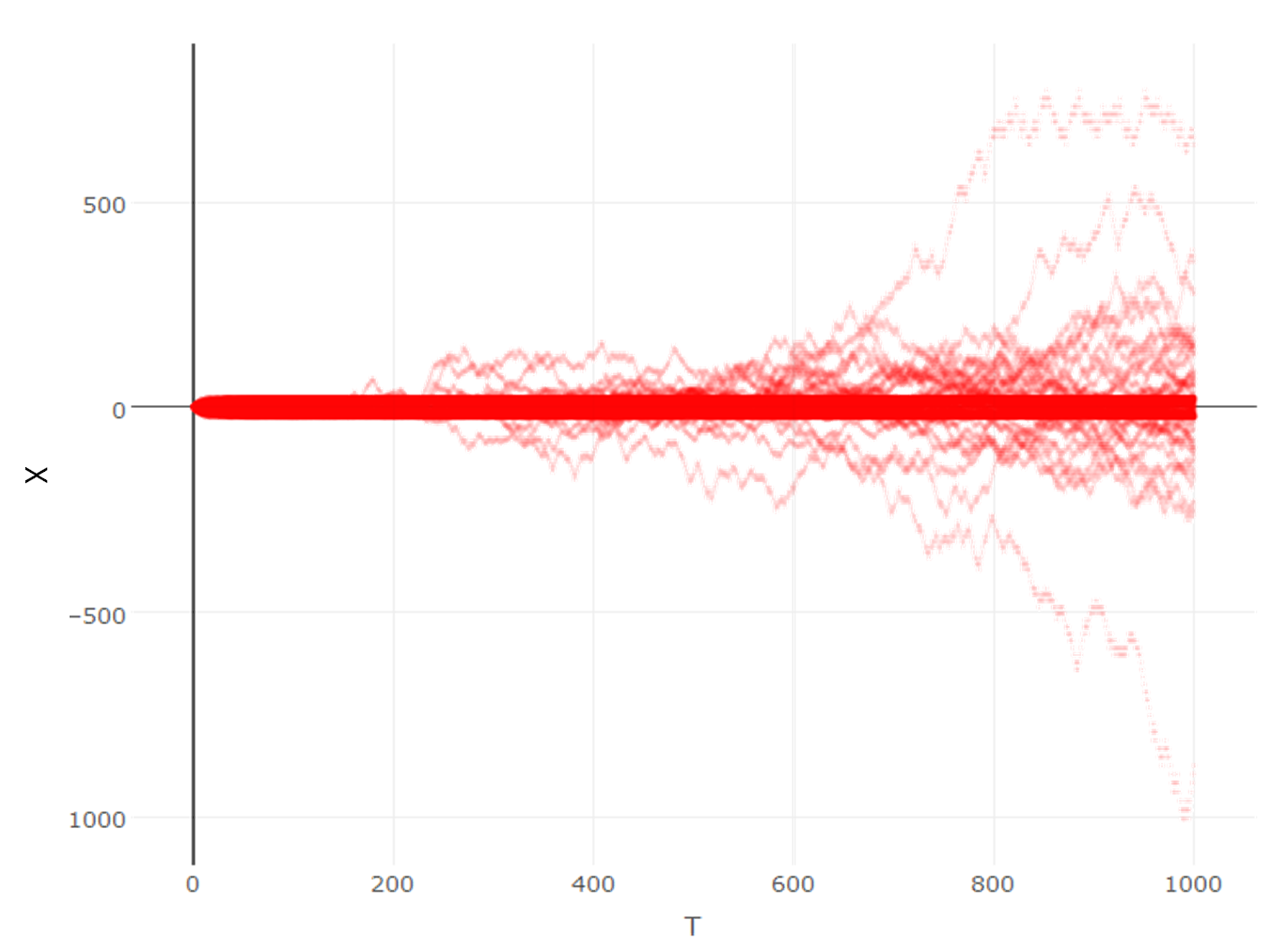}\\
\textbf{\footnotesize \noindent
(e). $\boldsymbol{\mu = 0, \sigma =  3.5 }$ without BGC \quad \quad \quad \quad (f). $\boldsymbol{\mu = 0, \sigma = 3.5 }$ with BGC
}\\
   \caption{Discretization Details of Simulations due to BGC with no Drift}
   \label{Fig:DiscretizationDetailsofSimulationsduetoBGCwithnoDrift}
\flushleft
\textbf{\footnotesize \noindent
(a). {\it \textbf{Negative}} diffusion is constrained in (b) the more it deviates away from the origin, causing {\it \textbf{downward}} diagonal bands to form, with {\it \textbf{some}} It\^{o} diffusions escaping the hidden reflective barrier since $\boldsymbol{\sigma}$'s effect is {\it \textbf{overtaken}}.\\
(c). {\it \textbf{Neutral}} diffusion is constrained in (d) the more it deviates away from the origin, causing horizontal bands to form, with {\it \textbf{no}} It\^{o} diffusions escaping the hidden reflective barrier since $\boldsymbol{\sigma}$'s impact is the {\it \textbf{smallest}}.\\
(e). {\it \textbf{Positive}} diffusion is constrained in (f) the more it deviates away from the origin, causing {\it \textbf{upward}} diagonal bands to form, with {\it \textbf{some}} It\^{o} diffusions escaping the hidden reflective barrier since $\boldsymbol{\sigma}$'s effect is {\it \textbf{overtaken}}.\\
We also note that in (b), (d) and (f), the transitions are exponentially `attracted' to the hidden (reflective) barrier.}
\end{figure}
%\FloatBarrier

\bigskip \noindent
From Figure \ref{Fig:DiscretizationDetailsofSimulationsduetoBGCwithnoDrift}, we can see that the BGC constrains the It\^{o} diffusions in the similar way as Figure \ref{Fig:DiscretizationDetailsofSimulationsduetoBGCwithnoDiffusion}, but as time continues to increase (or pass), then the impact of BGC diminishes.
This means that the reflective barrier nature of BGC becomes overtaken by the It\^{o} process itself when the magnitude of $\sigma$ forces the process to escape the barrier.

\bigskip \noindent
To statistically assess the impact of BGC on the resulting distributions, the corresponding densities were plotted in Figure \ref{Fig:ImpactofBGContheDiffusionAlteredDistributionofItoDiffusions}, the most noticable being (b).

%------SIGMA----------------------------------------------------------2
\begin{figure}
   \centering
   \includegraphics[scale=0.6]{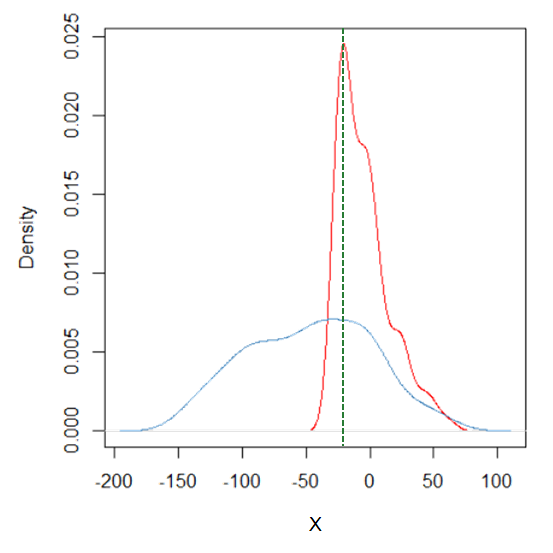}
   \includegraphics[scale=0.6]{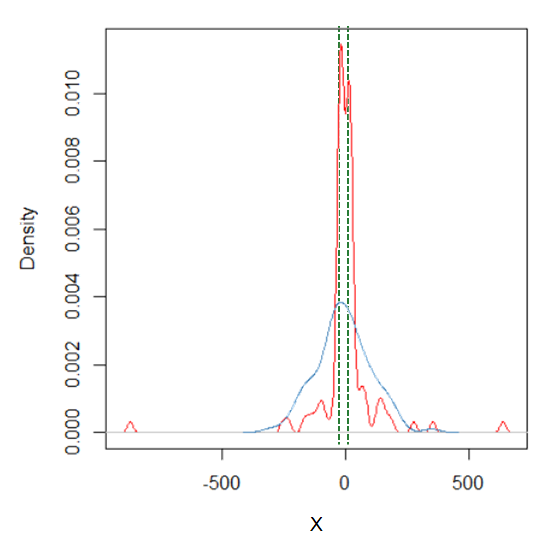}\\
\textbf{\footnotesize \noindent
(a). $\boldsymbol{\mu = 0, \sigma = -1.5}$  \quad \quad \quad \quad \quad \quad \quad \quad \quad (b). $\boldsymbol{\mu = 0, \sigma = 3.5}$}\\
   \includegraphics[scale=0.6]{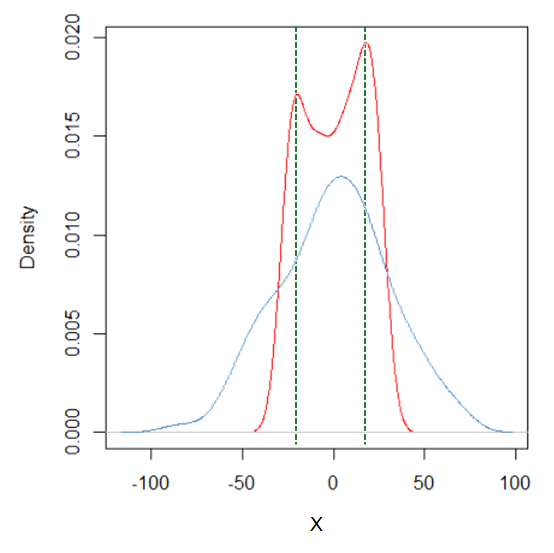}\\
\textbf{\footnotesize \noindent
(c). $\boldsymbol{\mu = 0, \sigma = 1}$}\\
   \caption{Impact of BGC on the Diffusion Altered Distribution of It\^{o} Diffusions}
   \label{Fig:ImpactofBGContheDiffusionAlteredDistributionofItoDiffusions}
\textbf{\footnotesize \noindent
Blue = original density $\quad$ , $\quad$ Red = BGC density $\quad$ , $\quad$ Green = peak.}\\
\flushleft
\textbf{\footnotesize \noindent
(a). BGC squeezes the {\it \textbf{negative}} skew distribution to the {\it \textbf{positive}} direction due to the impact of the hidden BGC barrier.\\
(b). BGC squeezes the {\it \textbf{positive}} skew distribution to the {\it \textbf{negative}} direction due to the impact of the hidden BGC barrier.\\
(c). BGC squeezes the {\it \textbf{zero}} skew distribution to both the {\it \textbf{negative}} and {\it \textbf{positive}} direction due to the impact of the hidden BGC barrier.
}
\end{figure}
%\FloatBarrier

%%%%%%%%%%%%%%%%%%%%%%%%%%%%%%%%%%%%%
%\newpage
\bigskip
\section{Conclusions}

\noindent
This paper has introduced the novel theory of Bi-Directional Grid Constrained (BGC) stochastic processes, where the further an It\^{o} diffusion drifts away from the origin, then the more it will be constrained.
The net effect of the BGC operates as a reflecting horizontal hidden barrier, from which we derived a theorem for the logarithmic bounds of the resulting envelope.
We have also shown how BGC effectively discretizes the It\^{o} diffusion paths into discrete bands, which is surprising.
We have also shown that if the diffusion parameter $\sigma$ is relatively large in comparison to the BGC $\Psi (X, t)$ function's magnitude, then the It\^{o} diffusions can escape or be transmitted from the hidden barrier.
These results are infinitely scalable to $n$-Dimensional It\^{o} diffusions, although the proof of which is reserved for future research.

\bigskip \noindent
There are immediate applications of this research, not only in finance (Taranto \& Khan, 2020) \cite{TarantoKhan2020a}, \cite{TarantoKhan2020b}, \cite{TarantoKhan2020c}, but also in many other fields.
Due to our It\^{o} process formulation, without requiring the typically high amount of parameters found in fields such as Physics and Economics, we find that our formulation supports greater portability to many fields.
One such typical application of BGC can be in the monetary policies of increasing quantitative easing as it can reduce or constrain the growth of unemployment or alternatively the growth of inflation, but after some time, the stimulus can end up having little or no effect (or worse, have an adverse effect).
It can also polarize or discretize the associated It\^{o} process that it is trying to constrain, such as displacing or marginalizing individuals or groups of individuals.
Future research on this topic can include the geometric classification of BGC functions and the estimation of the first passage time (FPT) for when the resulting BGC hidden barrier is hit.

\bigskip
%\bigskip
\bibliographystyle{amsplain}

\bigskip
\bigskip \noindent
\begin{center}
%Aldo Taranto, Shahjahan Khan, Ron Addie\\

%\smallskip
%\email{Aldo.Taranto@, Shahjahan.Khan@, Ron.Addie@, @usq.edu.au}

%\author[Shahjahan Khan]{Shahjahan Khan}
%\email{Shahjahan.Khan@usq.edu.au}

%\smallskip
%\address{School of Sciences, University of Southern Queensland,\\
%Toowoomba, QLD 4350, Australia}

\bigskip
\thanks{The first author was supported by an Australian Government Research Training Program (RTP) Scholarship.\\
\bigskip
We would like to thank Prof. Laura Sacerdote of Universit\`{a} Degli Studi Di Torino, for her invaluable advice on refining the early stages of the paper.
We would also like to thank the independent referees of this journal for their endorsements and suggestions.}
\end{center}

\end{document}